\documentclass{article}

\usepackage[utf8]{inputenc} 
\usepackage{amsfonts}
\usepackage{amsmath}
\usepackage{amsthm} 
\usepackage{amssymb}
\usepackage{color}
\usepackage{graphicx}
\usepackage{setspace}
\usepackage{enumerate}
\usepackage{booktabs}
\usepackage{pstricks}
\usepackage{url}
\usepackage{hyperref}
\usepackage{mathtools}
\usepackage{enumitem}
\usepackage{cancel}
\usepackage[ruled, boxed, noend, noline, linesnumbered, norelsize]{algorithm2e}
\usepackage[normalem]{ulem}
\usepackage{multicol}
\hypersetup{
	colorlinks,
	linkcolor={red!50!black},
	citecolor={blue!50!black},
	urlcolor={blue!80!black}
}

\newcommand{\norm}[1]{\left\|#1\right\|}
\newcommand{\interior}{\mathrm{int}\,}

\newcommand{\closure}{\mathrm{cl}\,}

\newcommand{\inProd}[2]{\langle #1 , #2 \rangle }

\newcommand{\dist}{ {\mathrm{dist}\,}}

\newcommand{\stdCone}{ {\mathcal{K}}}

\newcommand{\stdInt}{ {e}}

\newcommand{\matRank}{{\mathrm{ rank } \,}}	
\newcommand{\Diag}{{\mathrm{ Diag } \,}}
\newcommand{\diag}{{\mathrm{ diag } \,}}
\newcommand{\conv}{{\mathrm{ conv } \,}}

\newcommand{\Aut}{{\mathrm{ Aut } \,}} 
\newcommand{\supp}{{\mathrm{ supp } \,}} 
\newcommand{\jAlg}{\mathcal{E}}
\newcommand{\jFr}{\mathcal{J}}
\renewcommand{\Re}{\mathbb{R}}

\newcommand{\jProd}[2]{ {#1 \circ #2 } }

\newcommand{\Sd}[1]{{#1}^{\downarrow}}
\newcommand{\Pe}[1]{\mathcal{P}^{#1}}
\newcommand{\csub}{\partial _C}

\newcommand{\hsub}{\partial^{\infty}}
\newcommand{\msub}{\Diamond}
\newcommand{\idem}{\mathcal{I}}
\newcommand{\ReIm}{\overline{\Re}}
\def\dom{{\rm dom}\,}
\newcommand{\spec}[1]{F}

\DeclarePairedDelimiter\abs{\lvert}{\rvert}%

\renewcommand{\S}{\mathcal{S}}                    
\newcommand{\tr}{\mathrm{tr}\,}    
\newcommand{\RNum}[1]{\uppercase\expandafter{\romannumeral #1\relax}}
\newtheorem{definition}{Definition}
\newtheorem{lemma}[definition]{Lemma}
\newtheorem{proposition}[definition]{Proposition}
\newtheorem{example}[definition]{Example}

\newtheorem{theorem}[definition]{Theorem}
\newtheorem{remark}[definition]{Remark}

\usepackage{ifthen}
\newcommand{\COMM}[2]{{
\ifthenelse{\equal{#1}{AT}}{\color{red}}{
\ifthenelse{\equal{#1}{BFL}}{\color{blue}}}
[#1: #2]
}}

\title{Generalized subdifferentials  of spectral functions over Euclidean Jordan algebras}
\author{
Bruno F. Louren\c{c}o%
	\thanks{Department of Statistical Inference and Mathematics, Institute of Statistical Mathematics, 10-3 Midori-cho, Tachikawa, Tokyo 190-8562, Japan.
		(\texttt{bruno@ism.ac.jp})}
\and 
Akiko Takeda\thanks{
Department of Creative Informatics, Graduate School of Information Science and Technology,
University of Tokyo, Tokyo, Japan and  RIKEN  Center  for  Advanced  Intelligence  Project,  1-4-1, Nihonbashi, Chuo-ku, Tokyo  103-0027, Japan. (\texttt{takeda@mist.i.u-tokyo.ac.jp}) 
}
}

\begin{document}
	\maketitle
\begin{abstract}
This paper is devoted to the study of generalized 
subdifferentials of spectral functions over Euclidean Jordan algebras. 
Spectral functions appear often in optimization problems playing 
the role of ``regularizer'', ``barrier'', ``penalty function'' and many others.	
We provide formulae for the regular, approximate and horizon subdifferentials of spectral functions. 
In addition, under local lower semicontinuity, we also furnish a formula for the Clarke subdifferential, thus extending an earlier result by Baes. As 
 application, we compute the generalized 
subdifferentials of the function that maps an element to its $k$-th largest 
eigenvalue. Furthermore, in connection with recent approaches for nonsmooth optimization, we present a study of the Kurdyka-{\L}ojasiewicz (KL) property
for spectral functions and prove a transfer principle for the KL-exponent.
In our proofs, we make extensive use of recent tools such as the commutation principle of Ram\'irez, Seeger and Sossa and majorization principles developed by Gowda.

\end{abstract}

	\section{Introduction}	
Let $f:\Re^r \to \ReIm$ be a function that is \emph{symmetric}, i.e., $f(u)$ does not change if we permute the coordinates of $u \in \Re^r$. 	
Here, $\ReIm$ denotes the extended line $[-\infty,+\infty]$.
Now, let us consider a Euclidean Jordan algebra $\jAlg$ of rank $r$, for example, 
the $r\times r$ symmetric matrices.
Then,  $f$ can be extended in a natural fashion to a function $\spec{f}$ over $\jAlg$ by defining for all $x \in \jAlg$
\[
\spec{f}(x) \coloneqq f(\lambda(x)),
\]
where $\lambda(x)\in \Re^r$ is the vector containing the eigenvalues of $x$ in nonincreasing order, i.e.,
\[
\lambda_1(x) \geq \cdots \geq \lambda_r(x).
\]
We call $F$ the \emph{spectral function induced by $f$}. 
Because $f$ is symmetric, it is known from the works of Baes \cite{B07}, Sun and Sun \cite{SS08}, Jeong and Gowda \cite{JG17} and others that several properties of $f$ are transferred to $\spec{f}$. 
For example, $f$ is convex if and only if $\spec{f}$ is convex. The same goes for differentiability.
Results of this type are sometimes called \emph{transfer results} or \emph{transfer principles}, e.g., \cite{JG17}.

Spectral functions are ubiquitous throughout optimization and recognizing that $F$ is a spectral function can make computing derivatives/subdifferentials of $F$ significantly simpler than if one tries to do so by scratch.
This is because transfer principles usually come with formulae that relate the derivatives/subdifferentials of $F$ and $f$.

Motivated by the needs of nonsmooth optimization, our goal in this paper is to obtain formulae for the regular, approximate and horizon subdifferentials of spectral functions without any extra assumptions such as local Lipschitzness. 
In nonsmooth optimization, the regular and approximate subdifferential are often used to express optimality conditions and in the analysis of algorithms. Also, conditions involving the horizon subdifferential are quite common to ensure that the function satisfies some desirable property. 
We will also obtain a formula for the Clarke subgradient 
with the assumption of local lower semicontinuity, which extends an 
earlier result by Baes~\cite{B06phd}.
We will use these formulae to compute the generalized subdifferentials of the eigenvalue functions in the context of Euclidean Jordan algebras, see Section~\ref{sec:eig}.

Another motivation comes from the so-called \emph{composite optimization}, where we wish to solve the problem
\begin{equation}
\min _{x \in \jAlg}\quad \Phi(x) = \psi(x) + F(x), \tag{OPT} \label{eq:opt}
\end{equation}
and only $\psi:\jAlg \to \Re$ is assumed to be smooth. It is 
common for the function $F$ to play the role of a ``regularizer'', ``penalty'' or ``barrier''. In those cases, $F$ is often a 
spectral function. Here 
are a few examples. In what follows, for $u \in \Re^r$, we denote its $p$-norm by $\norm{u}_p$ and the sum of the $\ell$ components with largest 
absolute value by $\abs{\norm{u}}_\ell$. 
\begin{align*}
F_1(x) = \mu\norm{\lambda(x)}_p,   & \qquad  F_2(x) =  -\mu \log \det(x),\\
F_3(x) = \mu(\norm{\lambda(x)}_1 - \abs{\norm{\lambda(x)}}_{\ell}),  & \qquad  F_4(x) =  \mu\, \matRank(x),
\end{align*}
where $\mu$ is a positive parameter. 
When $p = 1$, $F_1$ is the $l_1$ regularizer. $F_2$ is a multiple of the classical self-concordant barrier for the symmetric cone associated to  $\jAlg$. The function $F_3$ maps $x$ to the sum of the  $r-\ell$ eigenvalues of $x$ with smallest absolute value, which is an important function for dealing with rank constrained problems, see~\cite{GS10} and Section~4 in \cite{GTT18}. Here, we are expressing $F_3$ as a DC (difference of convex) function.
We observe that $F_1,F_2,F_3,F_4$ are all spectral functions, while 
$F_3$ and $F_4$ are nonsmooth and nonconvex.
In any case, under appropriate regularity conditions, a necessary condition for $x^*$ to be a local optimal solution 
to \eqref{eq:opt} is that 
\[
-\nabla\psi(x^*) \in \partial F(x^*),
\]
where $\partial F(x^*)$ is the approximate subdifferential of 
$F$ at $x^*$, see Exercise 8.8 and Theorem 8.15 in \cite{RW}.

Yet another motivation for this work is that the approximate subdifferential  is necessary in order to compute the so-called \emph{Kurdyka-{\L}ojasiewicz (KL) exponent}, which has been shown to control the convergence properties of many first-order methods as can be seen, for instance, in the classical work by Attouch, Bolte, Redont and Soubeyran \cite{ABRS10}.
For a recent discussion on this topic, see the work by Li and Pong \cite{LP18}.

While there are many criteria that can be used to show that a function satisfies the so-called \emph{KL-property}, it is often highly nontrivial to compute the KL-exponent \cite{LP18}. For instance, if we wish to compute 
the $KL$-exponent of $\Phi$, we have to analyze the approximate subdifferentials of $F$, because $\partial \Phi(x) = \nabla \psi(x) + \partial F(x)$, as can be seen in Exercise~8.8 of \cite{RW}.
In this paper, although we will not compute the KL-exponent of $\Phi$ itself,
as an application of our results, we will show that if $f$ is a symmetric function and $F$ is the corresponding spectral function, then $f$ and $F$ share the same KL-exponent. Admittedly, this is not a very powerful result, but it seems to be beyond what can be proved directly with the results of \cite{LP18} (see Remark~\ref{rem:comp}) and we believe it is a first step towards a more comprehensive study of the KL-exponent of composite functions where one of the functions is spectral.  
%
%
%
\subsection{Previous works} 
Lewis \cite{Le96,Le96_2,Le99} has discussed extensively the case of 
spectral functions over symmetric real matrices and Hermitian complex matrices. In particular, in \cite{Le99}, Lewis gave expressions for 
the regular, approximate and horizon subdifferentials  of 
spectral functions over symmetric real matrices. A formula 
for Clarke subdifferentials was also given for the locally Lipschitz case.

Spectral functions over the algebra associated to the second order cone  
were initially studied by Fukushima, Luo and Tseng \cite{FLT02} and 
by Chen, Chen and Tseng~\cite{CCT04}. In \cite{CCT04}, there is a discussion of the Clarke subdifferential of locally Lipschitz spectral functions and 
Sendov \cite{Se07} gave  formulae for regular, approximate and horizon subdifferentials. 
Sendov also proved a formula for the Clarke subdifferential under the hypothesis of local lower semicontinuity.

In the general framework of Euclidean Jordan algebras, Baes \cite{B06phd,B07},  Sun and Sun \cite{SS08} and Jeong and Gowda \cite{JG16,JG17} proved several key results 
regarding spectral functions and the related notion of 
spectral sets. 
However, as far as we know, until now there were no results for the regular, approximate and horizon subdifferentials 
of spectral functions. Furthermore, results for the Clarke subgradient were 
only known in the locally Lipschitz case.
Related to Clarke subgradients, we mention in passing that Kong, Tun\c{c}el and Xiu proved an expression for the Clarke subgradient of the orthogonal projection of the symmetric cone associated to a Euclidean Jordan algebra~\cite{KTX09}.


\subsection{Contributions of this work}
In this work, we have three contributions.
The first is a meta-formula for the generalized subdifferentials of a 
spectral function. We will show that if $F: \jAlg \to \ReIm$ is a 
spectral function induced by $f:\Re^r \to \ReIm$, then
there is a formula that relate the generalized subdifferentials 
of $F$ and $f$, see Theorems~\ref{theo:main}, \ref{theo:hull} and \ref{theo:clarke}.

A feature of our results is that we will never assume that the algebra $\jAlg$  is simple, which makes some results more general, but a bit harder to 
prove. Every Jordan algebra can be decomposed as a direct sum of simple algebras and simplicity is, in many cases, a harmless hypothesis. Previous work by Lewis~\cite{Le99} and Sendov~\cite{Se07} can be seen as containing results for specific cases of simple Euclidean Jordan algebras.
However, because the generalized subdifferentials do not behave nicely with respect to partial subdifferentiation, there are cases where we cannot extend 
results from simple to general Euclidean Jordan algebras in a straightforward way. 
We emphasize that our results are directly applicable to a situation where, 
for example, $\jAlg$ is a direct product $\S^{r_1}\times \cdots \times \S^{r_{\ell}}$, where $\S^r$ denotes the space of $r\times r$ real symmetric matrices.

Our second contribution is providing formulae for the generalized 
subdifferentials of the function $\lambda_k: \jAlg \to \Re$, which maps 
an element $x \in \jAlg$ to its $k$-th largest eigenvalue, see Theorem~\ref{theo:eig}. We believe this is the first time such formulae are given in the context of Euclidean Jordan algebras.

Last, we will show a transfer principle of the KL-property for spectral 
functions and show that $F$ and $f$ must share the same KL-exponent, see 
Theorem~\ref{theo:kl}.

This work is divided as follows. In Section~\ref{sec:prel}, we review generalized subdifferentials. In Section~\ref{sec:jordan}, we overview the necessary concepts from the theory of Euclidean Jordan algebras. 
In Section~\ref{sec:trans}, we develop and present our main results regarding generalized subdifferentials of spectral functions. Finally, in Section~\ref{sec:kl} we discuss the KL-property and KL-exponent of spectral functions.

\section{Preliminaries}\label{sec:prel}
\subsection{Notation}
Given an element $u \in \Re^r$, we will denote its $i$-th component 
by $u_i$. We write $\Re^r_\geq$ for the cone of elements $u$ satisfying 
$u_1 \geq \cdots \geq u_r$.  We write $\Re^r_+$ for the nonnegative orthant, i.e., the elements $u \in \Re^r$ such that $u_i \geq 0$ for every $i$.
We will write $\Pe{r}$ for the \emph{group of 
$r\times r$ permutation matrices}. Given $u \in \Re^r$, we write 
$\Pe{r}(u)$ for the stabilizer subgroup of $u$, i.e., 
\[
\Pe{r}(u) \coloneqq \{P \in \Pe{r} \mid P(u) = u \}.
\]
The convex hull, the interior and the closure of a set $C$ will be denoted by $\conv C$, $\interior C$ and $\closure C$, respectively.
If $f:\Re^r\to \ReIm$ is a function, the domain of $f$ (i.e., the elements 
for which $f$ is finite) will be denoted by $\dom f$. 
We assume that $\Re^r$ is furnished with the usual Euclidean inner product $\inProd{\cdot}{\cdot}$ and the usual Euclidean norm $\norm{\cdot}$.
\subsection{Generalized subdifferentials}\label{sec:rev}
In this subsection, we  recall a few notions of generalized subdifferentials. However, the discussion on the Clarke subdifferential will be 
postponed until Section~\ref{sec:conv}.
Let $f: \Re^r \to \ReIm$ be a function and $u \in \dom f$.
We say that $d$ is a \emph{regular subgradient} 
of $f$ at $u$ if 
\begin{equation}\label{eq:reg_sub}
\liminf_{\substack{v\to 0 \\ v\neq 0} } \frac{f(u+v) - f(u) - \inProd{d}{v}}{\norm{v}} \geq 0.
\end{equation}
The set of regular subgradients of $f$ at $u$ is denoted by $\hat \partial f(u)$ and is called the \emph{regular subdifferential} of $f$ at $u$.
From \eqref{eq:reg_sub} it follows that $d \in \hat \partial f(u)$ if and only if for every $\epsilon > 0$ there exists some $\delta > 0$ such that $\norm{v} \leq \delta$ implies 
\begin{equation}\label{eq:reg_sub2}
f(u+v) - f(u) - \inProd{d}{v} \geq -\epsilon\norm{v}
\end{equation}
We say that $d$ is an \emph{approximate subgradient} (also called \emph{limiting subgradient}) of $f$ at $u$ if
there are sequences $\{u^k\}$, $\{d^k\}$ such that every $k$ satisfies $d^k \in \hat \partial f(u^k)$ and the following limits hold:
\[
u^k \to u, \qquad f(u^k)\to f(u), \qquad d^k \to d.
\]
The set of approximate subgradients of $f$ at $u$ is denoted by 
$\partial f(u)$ and is called the \emph{approximate subdifferential} of 
$f$ at $u$. 

We say that $d$ is an \emph{horizon subgradient} of 
$f$ at $u$  if
there are sequences $\{u^k\}$, $\{d^k\}$, $\{t^k\}$ such that every $k$ satisfies $ d^k \in \hat \partial f(u^k)$ and the following limits hold:
\[
u^k \to u, \qquad f(u^k)\to f(u), \qquad t^kd^k \to d, \qquad t^k \downarrow 0.
\]
Here, $ t^k \downarrow 0$ indicates that all the $t^k$ are nonzero and that $t^k$ is a monotone nonincreasing sequence converging to zero.
The set of horizon subgradients, called the \emph{horizon subdifferential}, will be denoted 
by $\hsub f(u)$. In variational analysis, conditions involving the horizon subdifferential are quite common, e.g., see Corollary 10.9 in \cite{RW}.
See also Section~8.B in \cite{RW} for examples of the subdifferentials discussed so far. 


We will also make use of the following characterization of regular subgradients. 
\begin{proposition}[Rockafellar and Wets, Proposition 8.5 in \cite{RW}]\label{prop:rw}
	Let $d \in \Re^r$. Then,	$d \in \hat \partial f(u)$ if and only if, 	
	on some neighborhood $U$ of $u$ there  exists a $C^1$ function $h:U \to \Re$ such that 
	\begin{align*}
	h(u) &= f(u), \quad \nabla h(u) = d\\
	h(v) &\leq f(v), \quad \forall v \in U.
	\end{align*}
\end{proposition}
In this paper,  sometimes we will prove results that are valid for several 
different notions of subdifferential. In that case, we use the symbol $\msub$ as a placeholder for some unspecified subdifferential, e.g., see Theorem~\ref{theo:main}. 
\section{Euclidean Jordan algebras}\label{sec:jordan}
Here, we give a brief overview of Jordan algebras and review the necessary tools to prove our results. More details can be found in Faraut and Kor\'anyi's book \cite{FK94} or in the survey by Faybusovich \cite{FB08}.
First of all, a Euclidean Jordan algebra $(\jAlg,\circ)$ is a finite dimensional real vector space 
$\jAlg$ equipped with a bilinear product $\circ:\jAlg\times \jAlg\to \jAlg$ and an inner product $\inProd{\cdot}{\cdot}$
satisfying the following properties:
\begin{enumerate}[label=$(\arabic*)$]
	\item $\jProd{x}{y} = \jProd{y}{x}$,
	\item $\jProd{x}({\jProd{x^2}{y}}) = \jProd{x^2}({\jProd{x}{y}})$, where $x^2 = \jProd{x}{x}$,
	\item $\inProd{\jProd{x}{y}}{z} = \inProd{x}{\jProd{y}{z}}$,
\end{enumerate}
for all $x,y, z \in \jAlg$. We can always assume that a Euclidean Jordan algebra has an element $e$ that satisfies $\jProd{e}{x} = x$, for all $x \in \jAlg$. Such an element $e$ is called the \emph{identity element}.
An element $c \in \jAlg$ satisfying $c^2 = c$ is called an \emph{idempotent}. 
A nonzero idempotent $c$ that cannot be written as the sum of two nonzero idempotents $\hat c, \tilde c$ satisfying $\jProd{\hat c}{\tilde c} = 0$ is called a \emph{primitive idempotent}.

In a Euclidean Jordan algebra the following spectral theorem holds.
\begin{theorem} [Spectral Theorem, see Theorem III.1.2 in \cite{FK94}]\label{theo:spec}
	Let $(\jAlg, \jProd{}{} )$ be a Euclidean Jordan algebra and let $x \in \jAlg$. Then there are	 primitive idempotents $ [c_1, \dots, c_r]$ satisfying $c_1 + \cdots + c_r  = \stdInt$ and
	\begin{flalign*}
	& \jProd{c_i}{c_j} = 0 \, \,\,\,\, \qquad \qquad \text{ for } i \neq j, 
	\end{flalign*}
	and  unique real numbers $\alpha _1, \ldots, \alpha _r$ satisfying
	\begin{equation}		
	x = \sum _{i=1}^r \alpha _i c_i \label{eq:dec}.
	\end{equation}
\end{theorem}
The $r$ that appears in Theorem \ref{theo:spec} only depends on the algebra $\jAlg$ and is called the \emph{rank of $\jAlg$}.
The $\alpha _i$ in Theorem \ref{theo:spec} are called 
the \emph{eigenvalues of $x$}. Although unique, the eigenvalues of $x$ might be repeated and \emph{they are not necessarily in nonincreasing/nondecreasing order}.
We define the \emph{rank of $x$} as the number of nonzero $\alpha_i$'s appearing in \eqref{eq:dec}. 
The ordered set $[c_1,\ldots, c_r]$ in Theorem \ref{theo:spec} is called a 
\emph{Jordan frame for $x$}. 

Here, we are using the notation $[c_1,\ldots, c_r]$ instead of 
$\{c_1,\ldots, c_r\}$ to emphasize that the order of the elements is taken into account, so, for example, $[c_1,c_2]$ and $[c_2,c_1]$ are different ordered sets.
Although $x$ might have many different Jordan frames, the sum of primitive idempotents associated to some eigenvalue must be unique.
\begin{proposition}[Unique sum of primitive idempotents, see Theorems~III.1.1 and III.1.2 in \cite{FK94}
]\label{prop:unique_sum}
Let $x \in \jAlg$ and $[c_1,\cdots, c_r], [c_1',\cdots, c_r']$ be two 
Jordan frames for $x$. Suppose that 
\[
x = \sum _{i=1}^r \alpha _i c_i = \sum _{i=1}^r \alpha _i' c_i'.
\]
Then, for every $\alpha \in \Re$,  we have
\[
\sum _{i \text { with } \alpha _i = \alpha} c_i = \sum _{i \text{ with }\alpha _i' = \alpha} c_i'.
\]
\end{proposition}
We define the eigenvalue map $\lambda:\jAlg \to \Re^r_\geq$ as the map
satisfying 
\[
\lambda(x) \coloneqq (\lambda_1(x), \ldots, \lambda _r(x)),
\]
where $\lambda_1(x) \geq \cdots \geq \lambda _r(x)$.
Here, $\lambda_i(x)$ denotes the $i$-th largest eigenvalue of $x$.

The trace map $\tr:\jAlg \to \Re$ is defined as 
\[
\tr(x) \coloneqq \lambda_1(x) + \cdots + \lambda _r(x).
\]
In fact, the trace map is a linear function.
Furthermore, it can be shown that the function that maps 
$x,y \in \jAlg$ to $\tr(\jProd{x}{y})$ is an inner product satisfying
Property $(3)$ of the definition Euclidean Jordan algebras.
Henceforth, we shall assume that the inner product $\inProd{\cdot}{\cdot}$ is 
given by 
\begin{equation}\label{eq:inner}
\inProd{x}{y} = \tr (\jProd{x}{y}),\quad \forall x,y \in \jAlg.
\end{equation}
Under this inner product, $\tr(x) = \inProd{e}{x}$ for all $x \in \jAlg$ and
 elements of any Jordan frame are mutually orthogonal. That is, if $\jFr = [c_1,\ldots, c_r]$ is Jordan frame, then $\inProd{c_i}{c_j} = 0$ if $i \neq j$.

The norm induced by $\inProd{\cdot}{\cdot}$ is given by 
\begin{equation*}
\norm{x} = \sqrt{\tr(x^2)}. \label{eq:norm_jordan}
\end{equation*}
With that, any primitive idempotent $c$ satisfies $\norm{c} = 1$.
Furthermore, the  map $\lambda$ becomes a Lipschitz continuous function with Lipschitz constant $1$, when $\Re^r$ is equipped with the usual Euclidean norm.
We now  summarize some important properties of $\lambda$.
\begin{lemma}[Properties of the eigenvalue map]\label{lem:eig}
Let $\jAlg$ be a Euclidean Jordan algebra of rank $r$ and let
$\lambda : \jAlg \to \Re^r_{\geq}$ be the eigenvalue map.
The following properties hold.
\begin{enumerate}[label=$(\roman*)$]
	\item $\norm{\lambda(x)-\lambda(y)} \leq \norm{x-y}$ holds, for all 
	$x,y \in \jAlg$. 
	\item For every $x \in \jAlg$, $\lambda$ has directional derivatives along all directions. Furthermore, letting $\lambda'(x;z)$ denote the directional derivative of $\lambda$ at $x$ along $z$, the following limit holds
\[
	\lim _{z\to 0} \frac{\lambda(x+z)-\lambda(x) - \lambda'(x;z) }{\norm{z}} = 0,
\]
where $\lambda'(x;z) \coloneqq \lim _{t\to 0} \frac{\lambda(x+tz)-\lambda(x)}{t}$.
\end{enumerate}
\end{lemma}
\begin{proof}
	\begin{enumerate}[label=$(\roman*)$]
			\item This was proved by Baes, see Corollary 24 in \cite{B07}.
			\item Baes showed that for every $i$, the function 
			$\lambda_i:\jAlg\to \Re$ that maps $x\in \jAlg$ to its $i$-th largest eigenvalue is directionally differentiable, see Theorem 36 in~\cite{B07}. 
			Therefore, all components of $\lambda$ are directionally differentiable, so  $\lambda$ must also be directionally differentiable. Then, it is a general fact that a Lipschitz continuous function that is directionally differentiable everywhere must also satisfy the limit above, see Lemma 2.1.1 and Remark 2.1.2 in \cite{HUL96}.
		\end{enumerate} 
\end{proof}

\subsection{Simultaneous diagonalization}\label{sec:diag}
Let $\jAlg$ be a Euclidean Jordan algebra of rank $r$. 
Given $x \in \jAlg$, we denote by $L_x:\jAlg \to \jAlg$ the \emph{Lyapunov operator associated to $x$}, which is the linear map satisfying
\[
L_x(z) = \jProd{x}{z}, \quad \forall z \in \jAlg.
\]
Given another element $y \in \jAlg$, we say that 
$x$ and $y$ \emph{operator commute} if
\[
L_xL_y = L_yL_x
\]
holds. 
It is known that $x$ and $y$ operator commute if and only if they share a common Jordan frame $\mathcal{J}$, see Lemma~X.2.2 in \cite{FK94}. This means that there are $r$ mutually orthogonal primitive idempotents $\jFr = [c_1,\cdots, c_r]$ such that $c_1+\cdots + c_r = \stdInt$ and 
\begin{align*}
x  = \sum _{i=1}^r a_i c_i, \qquad y  = \sum _{i=1}^r b_i c_i,
\end{align*} 
where the $a_i$ and $b_i$ are the eigenvalues of $x$ and $y$, respectively.
More generally if $\jFr$ is a Jordan frame for which $x \in \jAlg$ can be expressed as linear combination of the $c_i$, we say that 
\emph{$\jFr$ diagonalizes $x$}. Therefore, the existence of a common Jordan frame for $x$ and $y$ means that $x$ and $y$ are \emph{simultaneously diagonalizable}.

Here, the $a_i$ and $b_i$ that appear in the decomposition of $x$ and $y$ are not necessarily sorted in nondecreasing/nonincreasing order.
However, reordering the $c_i$, we may suppose that 
the $a_i$ are sorted in an nonincreasing order, i.e., 
$a_i = \lambda _i(x)$, for all $i$. With respect to this new ordering, we can write
\begin{align*}
x  = \sum _{i=1}^r \lambda_i(x) c_i, \qquad y = \sum _{i=1}^r \tilde {b}_ic_i,
\end{align*} 
where $[\tilde b_1,\ldots, \tilde b_r]$ is some permutation of 
$[b_1, \ldots, b_r]$. Because the idempotents in $\jFr$ are orthogonal amongst themselves,  we have for every $i$
\[
\inProd{c_i}{x} = \lambda_i(x), \qquad \inProd{c_i}{y} = \tilde b_i.
\]
With that in mind, we are going to introduce the function $\diag(\cdot,\jFr): \jAlg \to \Re^r$, which maps an element $z \in \jAlg$ to its ``diagonal'' with respect the Jordan frame $\jFr$. That is, we have
\[
\diag(z,\jFr) = (\inProd{c_1}{z},\ldots, \inProd{c_r}{z}), \quad \forall z \in \jAlg.
\]
If $\jFr$ is a frame that diagonalizes $z$, then $\diag(z,\jFr)$ is, in fact, the 
eigenvalue vector of $z$. Of course, $\diag(z,\jFr)$ might not be sorted in any particular way.
However, for the specific $x$ and $y$ we have discussed so far, 
we have
\[
\diag(x,\jFr) = \lambda(x), \qquad \diag(y,\jFr) = (\tilde b_1, \ldots, \tilde b_r).
\]

We are now going to introduce two more extra notations. We will 
denote by $\jFr(x,y)$ the set of common Jordan frames $\jFr$ for $x,y$ for which 
$\diag(x,\jFr) = \lambda(x)$. In other words,
not only $\jFr$ must be a common Jordan for $x$ and $y$, but it must also 
be such that the eigenvalues of $x$ appear in nonincreasing order.
Here, we emphasize that the eigenvalues of $y$ might appear in no particular order.
 By convention, if $x$ and $y$ do not operator commute, we will define $\jFr(x,y) = \emptyset$. 
We observe that since $L_{\alpha y + \beta z} = \alpha L_y +\beta L_z$, we have
\begin{equation}\label{eq:simultaneous}
\jFr(x,y) \neq \emptyset \text{ and } \jFr(x,z) \neq \emptyset \,\,\Rightarrow\,\, \jFr(x,\alpha y + \beta z ) \neq \emptyset,\, \forall \alpha, \beta \in \Re. 
\end{equation}
 
Furthermore, we will define $\jFr(x) \coloneqq \jFr(x,x)$. That is,
 $\jFr(x)$ is the set of Jordan frames of $x$ for which the eigenvalues of $x$ appear in 
 nonincreasing order. We have $\jFr(x,y) \subseteq \jFr(x)$ for every $x,y \in \jAlg$.

We also need a map that plays the opposite role of $\diag(\cdot, \jFr)$.
Let $\Diag(\cdot,\jFr): \Re^r \to \jAlg$ be the map that takes a vector 
in $\Re^r$ and constructs a ``diagonal element'' in $\jAlg$ according to 
$\jFr$, i.e., 
\[
\Diag(u,\jFr) \coloneqq \sum _{i=1}^r u_i c_i. 
\]
We have $\diag(\Diag(u,\jFr),\jFr) = u$, for every $u \in \Re^r$.
We observe that, since $[c_1,\ldots, c_r]$ is a Jordan frame,  
the eigenvalues of $\Diag(u,\jFr)$ are precisely the $u_i$. 

\subsection{The directional derivative of the $i$-th largest eigenvalue}\label{sec:eig_dir}
In this section, we will describe an expression proved by Baes \cite{B07} to compute 
the directional derivative of the $i$-th largest eigenvalue. For that, 
we need to review the \emph{Peirce decomposition}, the properties of 
\emph{quadratic maps} in Euclidean Jordan algebras and, most regrettably, 
introduce more notation.

Let $c$ be an idempotent and $\alpha \in \Re$.
We define
\[
V(c,\alpha) \coloneqq \{x \in \jAlg \mid \jProd{c}{x} = \alpha x \}.
\]
Now, let $x \in \jAlg$ be an arbitrary element (not necessarily an idempotent), 
the \emph{quadratic map of $x$} is the linear map $Q_x: \jAlg \to \jAlg$ such that 
\[
Q_x(y) = 2\jProd{x}{(\jProd{x}{y})} -  \jProd{(\jProd{x}{x})}{y}.
\]
$Q_x$ is always self-adjoint.
With that, we have the following result.
\begin{theorem}[Peirce Decomposition, see Proposition IV.1.1 and page 64 in \cite{FK94}]\label{theo:peirce1}
	Let $(\jAlg, \jProd{}{})$ be an Euclidean Jordan algebra of rank $r$ and let
	$c \in \jAlg$ be an idempotent of rank $\ell$. Then $\jAlg$ is decomposed as the orthogonal direct sum
\[
	\jAlg = V(c,1) \bigoplus V\left(c,\frac{1}{2}\right) \bigoplus V(c,0).
\]
	In addition, $(V(c,1),\jProd{}{})$ and ${(V(c,0),\jProd{}{})}$ are  Euclidean Jordan algebras of rank $\ell$ and $r - \ell$, respectively.
	The orthogonal projections on $V(c,1)$ and $V(c,0)$ are given by 
	$Q_c$ and $Q_{\stdInt - c}$, respectively.
\end{theorem}
Next, we move on to the necessary notation. 
The eigenvalues of $x$ might be repeated so, for instance, it could be the 
case that $\lambda _3(x) = \lambda _4(x) = \lambda _5(x)$. The next
notation  corresponds to 
a way of assigning the indices $3,4,5$ to $1,2,3$.
That is, we need to map an index $i$ to its ``relative position'' with respect to the eigenvalues of $x$ that are equal to $\lambda _i(x)$.
Here, we will mostly follow the notation proposed by Baes in \cite{B07} and define 
for every $p \in \{1,\ldots, r\}$, the integer $l_p(x)\geq 1$ which is such that 
\[
\lambda_1(x) \geq \cdots \geq \lambda _{p - l_p(x)}(x) > \lambda _{p - l_p(x)+1}(x) = \cdots = \lambda _p(x) \geq \cdots \geq \lambda _r(x).
\]
Furthermore, if  $x  = \sum _{i=1}\lambda _i(x)c_i \in \jAlg$
we will denote by $e_p(x)$ the sum of the $c_i$ satisfying $\lambda _i(x) = \lambda _p(x)$, i.e., 
\[
e_p(x) = \sum _{i \text{ with } \lambda_i(x) = \lambda_p(x)} c_i.
\]
We  remark that  $\mathsf{f}'_p(x)$ was used instead of $e_p(x)$ in \cite{B07}. 
\begin{example}
Suppose that the rank of $\jAlg$ is $r = 7$ and the eigenvalues of $x  \in \jAlg$ are as follows.
\[
\lambda _1 > \lambda _2 = \lambda _3 = \lambda _4 > \lambda _5 = \lambda _ 6 > \lambda _7.
\]
Then $l_1 = l_7 = 1$, because $\lambda _1$ and $\lambda _7$ are unique eigenvalues. 
We have $l_2 = 1$, $l_3 = 2$ and $l_4 = 3$, since $\lambda _2,\lambda _3,\lambda_4$ are, respectively, the ``first'', ``second'' and ``third''  eigenvalues of a group of three equal eigenvalues.  Similarly,  
we have $l _5 = 1$ and $l_6 = 2$.

We have $e_1(x) = c_1$, $e_7(x) = c_7$,
\[
e_2(x) = e_3(x) = e_4(x) = c_2+c_3+c_4, \qquad e_5(x) = e_6(x) = c_5+c_6.
\]
\end{example} 
Finally, let $\jAlg' \subseteq \jAlg$ be an Euclidean Jordan algebra and let $x \in \jAlg'$. Then, the eigenvalues of $x$ as an element of $\jAlg'$ might be different from the eigenvalues of $x$ seen as an element of $\jAlg$. 
When it is necessary to make this distinction, we will denote the \emph{$i$-th eigenvalue of 
$x$ seen as element of $\jAlg'$} by 
\[
\lambda _i(x,\jAlg').
\]
The eigenvalue map of the algebra $\jAlg'$ will be similarly 
denoted by $\lambda(\cdot, \jAlg')$.
We have now all pieces in place to state the following 
theorem.
\begin{theorem}[Baes, Theorem 36 in \cite{B07}]\label{theo:dir}
Let $x,z \in \jAlg$ and consider the spectral decomposition of 
$x$:
\[
x = \sum _{i=1}^r \lambda _{i}(x) c_i.
\]
Then the directional derivative of the $i$-th largest eigenvalue of $x$ along 
the direction $z$ is given by
\[
\lambda_i'(x;z) 
= \lambda _{l_i(x)}(Q_{c}z, V(c,1)),
\]
where $c = e_i(x)$.
\end{theorem}
From Theorem~\ref{theo:peirce1}, $Q_{c}z$ is the projection of $z$ in the algebra 
$V(c,1)$. Therefore, to compute $\lambda_i'(x;z)$ we need to project $z$ 
on $V(c,1)$, and then compute the $l_i(x)$-th eigenvalue of the projection with 
respect the algebra $V(c,1)$, where $l_i(x)$ is the ``relative position'' of 
the index $i$ with respect to the eigenvalues of $x$ that are 
equal to $\lambda_{i}(x)$.

\subsection{Spectral functions and sets}
Let $\jAlg$ be a Euclidean Jordan algebra of rank $r$ and let $f: \Re^r \to \ReIm$ be a function. We say that $f$ is a \emph{symmetric function} if 
$f(Pu) = f(u)$ holds for every $u \in \Re^r$ and every permutation matrix $P \in \Pe{r}$. 
Symmetric functions satisfy the following key relation between 
subdifferentials:
\begin{equation}\label{eq:sym_sub}
\msub f(Pu) = P \msub f(u), \quad \forall u \in \Re^r, \forall P \in \Pe{r}, 
\end{equation} 
whenever $\msub$ is $\hat \partial, \partial$ or $\hsub$, e.g., Proposition~2 in \cite{Le99}.
We remark that \eqref{eq:sym_sub} will be used often in this paper.

We denote by $F:\jAlg \to \ReIm$ the \emph{spectral map induced by $f$}, which is the function defined as 
\[
F(x) \coloneqq f(\lambda(x)), \quad \forall x \in \jAlg.
\]
The function $F$ is well-defined, even if $f$ is not symmetric. However, if $f$ is indeed symmetric, many properties of $f$ are transferred to $F$.

There is also a notion of spectral set. We say that  $Q \subseteq \Re^r$ is a \emph{symmetric set} if $PQ = Q$ for every $P \in \Pe{r}$. Then the \emph{spectral set}
induced by $Q$ is defined as 
\[
\Omega \coloneqq \{x \in \jAlg \mid \lambda(x) \in Q \}.
\]

To conclude this subsection, we now move on to the notion of \emph{weakly spectral} sets/maps, which was introduced by Gowda and Jeong in \cite{GJ17}. We say that a linear bijection $A:\jAlg\to \jAlg$ is a \emph{Jordan algebra automorphism} if
\[
\jProd{Ax}{Ay} = A(\jProd{x}{y}), \quad \forall\, x,y \in \jAlg.
\]
The group of Jordan algebra automorphisms is denoted by $\Aut \jAlg$.
Then, a function $F: \jAlg \to \ReIm$ is said to be \emph{weakly spectral} if 
\[
F(Ax) = F(x), \quad \forall x \in \jAlg, \forall A \in \Aut \jAlg.
\]
 A set $\Omega \subseteq \jAlg$ is said to be \emph{weakly spectral} if $A\Omega = \Omega$ holds for every $A\in \Aut \jAlg$.
 {A spectral map/set must also be weakly spectral}, 
 but the converse is not true in general, see remarks in Section 3 of \cite{GJ17}.
\section{Transfer principles for generalized subdifferentials}\label{sec:trans}
We start with a description of our setting and a few conventions.
Throughout Sections~\ref{sec:trans} and \ref{sec:kl}, $(\jAlg,\jProd{}{})$ denotes a Euclidean Jordan algebra of rank $r$,  
the inner product of two elements of $\jAlg$ is given by \eqref{eq:inner} and the norm is the one induced by $\eqref{eq:inner}$.
Although we are using the same symbol to denote the Euclidean inner product on $\Re^r$ and the trace inner product on $\jAlg$, there will be no confusion. The letters $x,y,z,s$ will always be reserved for elements of $\jAlg$ and $u,v,d$ for elements of $\Re^r$. 

Let $F: \jAlg \to \ReIm$ be a spectral function induced by some symmetric function $f: \Re^r \to \ReIm$. 
Our first goal is to prove the following meta-formula:
\begin{equation}
\msub  F(x) = \{ s \in \jAlg \mid \exists \jFr \in \jFr(x,s) \text { with } \diag(s,\jFr) \in \msub f(\lambda(x))
\}, \tag{Transfer} \label{eq:g1}
\end{equation}
where $\msub$ is either $\hat \partial, \partial$, or $\hsub$.
\begin{remark}
For the sake of dispelling any possible confusion, $\msub f(\lambda(x))$ should be interpreted as 
$(\msub f)(\lambda(x))$, i.e., $\msub f(\lambda(x))$ is the generalized subdifferential of $f$ at $\lambda(x)$.
\end{remark}

Proving \eqref{eq:g_gen}  will require several tools old and new,  such as commutation principles \cite{RSS13,GJ17}, majorization principles \cite{Go17} and the formulae for the directional derivatives of the eigenvalue functions \cite{B07}. 
\subsection{Commutation principles and generalized subdifferentials}\label{sec:com}
The first step towards \eqref{eq:g1} is proving that 
if $F$ is a spectral function and $s$ is any generalized subgradient of $x$, then $x$ and 
$s$ must operator commute. For that, we 
will use a commuting principle proved 
by  Ram\'irez, Seeger and Sossa \cite{RSS13}. 
\begin{theorem}[Ram\'irez, Seeger and Sossa\footnote{
		Here, we are quoting the theorem as it appears in Gowda and Jeong's paper \cite{GJ17} (Theorem 1.1 therein), since it is more suited to our purposes.}  \cite{RSS13}]\label{theo:rss}
	Suppose that $\Omega \subseteq \jAlg$ is a spectral set and $F:\jAlg \to \Re$ is a spectral function. Let $\Theta:\jAlg \to \Re$ be Fr\'echet differentiable. If $x^*$ is a local minimizer/maximizer of 
\[
	{x \in \Omega} \, \mapsto \, \Theta(x) + F(x)
\]
	then $x^*$ and $\nabla\Theta(x^*)$ operator commute\footnote{We recall that  $x^*$ is a local minimum if  there exists a neighbourhood $\mathcal{V}$ of $x^*$ such that $\Theta(x^*)+F(x^*) \leq \Theta(x) + F(x)$ holds for every $x \in \Omega \cap \mathcal{V}$.}.
\end{theorem}
Recently, Gowda and Jeong showed that it is possible to weaken  the hypothesis of Theorem \ref{theo:rss} and consider
weakly spectral sets/functions instead \cite{GJ17}. 
\begin{theorem}[Gowda and Jeong \cite{GJ17}]\label{theo:gj}
	The conclusion of Theorem \ref{theo:rss} holds if $\Omega $ is a weakly spectral set and $F$ is a weakly spectral function.
\end{theorem}
Using the variational characterization of the regular subdifferential, we can prove the following new result, which 
is more general than what is strictly necessary for 
proving \eqref{eq:g1}, but we believe it is still useful.
\begin{proposition}[Operator commutativity for weakly spectral functions]\label{prop:p_com}
	Let $\spec{f}:\jAlg \to \ReIm$ be a weakly spectral function.
	Suppose
\[
	s \in \msub \spec{f}(x),
\]
	where $\msub$ is either $\hat \partial, \partial$ or $\hsub$. Then, $x$ and $s$ operator commute. 
\end{proposition}
\begin{proof}
	First, we prove the result for the case $s \in \hat \partial \spec{f}(x)$.
	By Proposition \ref{prop:rw}, there exists a $C^1$ function $H$ 
	such that $H(x) = F(x)$, $\nabla H(x) = s$  and $H(y) \leq F(y)$ for all $y$ near $x$. 
	We invoke Theorem~\ref{theo:gj} using $F$, $\Theta = -H$ and $\Omega = \dom F$. By the properties of $H$, we have that 
	$x$ is a local minimum of $\Theta + F = F - H$.
	Therefore, $x$ commutes with $\nabla \Theta(x) = -s$, so it must commute with $s$ too. In reality, there are some minor technical details we have  overlooked, see the footnote\footnote{The functions 
	in Theorem~\ref{theo:gj} are finite functions defined 
everywhere, whereas $F$ is an extended value function and $H$ is defined only in a neighbourhood of $x$. To sidestep this, we define 
$\hat F$ such that $\hat F (y) = F(y)$ if $y \in \dom F$ and 
$\hat F(y) = F(x)$ if $y \not \in \dom F$. With that, $\hat F$ is still a 
weakly spectral function. Next we need to extend $H$ to a function 
defined over $\jAlg$ which coincides with $H$ in some neighbourhood of $x$. It is a classical fact that this can always be done and  here we show briefly why. Suppose that $H$ is defined over some open set $\mathcal{U}$.
Let 
$\mathcal{V} \subseteq \mathcal{U}$ be an open ball such that $\closure \mathcal{V} \subseteq \mathcal{U}$ and over which $x$ is a local minimizer of $F-H$.
Next, pick any function $\psi$ that is smooth and such that 
$\psi$ is $1$ on the compact set $\closure \mathcal{V}$ and $0$ outside $\mathcal{U}$. Then, we 
define $\hat H$ by letting $\hat H (y) = \psi(y)H(y)$ if $y \in \mathcal{U}$ and 
$\hat H (y) = 0$ if $y \not \in \mathcal{U}$. With that, we have that $\nabla \hat H(x) = s$ and $x$ is a local minimum of $\hat F - \hat H$ restricted to $\dom F$. 
Then, as before, we can invoke Theorem~\ref{theo:gj} with $\hat F$, $\Omega = \dom F$ and $\Theta = -\hat H$. 
} below.
	
	Next, suppose instead that $s \in \partial \spec{f}(x)$ or $s \in \hsub \spec{f}$. Then, there are sequences $\{x^k\}$, $\{s^k\}, \{t^k\}$ such that every $k$ satisfies $s^k \in \hat \partial \spec{f}(x^k)$ and the following limits hold.
\[
	x^k \to x, \qquad \spec{f}(x^k)\to \spec{f}(x), \qquad t^k s^k \to s.
\]
	Here, there are two cases for $\{t^k\}$. If 
	$s \in \partial \spec{f}(x)$, then $t^k = 1$ for every $k$. If $s \in \hsub \spec{f}(x)$, then $t^k \downarrow 0$.
	
	Either way, because $s^k \in \hat \partial \spec{f}(x^k)$, from what we have proved so far, 
	we have that $s^k$ and $x^k$ operator commute for every $k$. That is, 
\[
	L_{s^k}L_{x^k} = L_{x^k}L_{s^k}, \qquad \forall k.
\]
	By taking limits, we conclude that $L_{s}L_{x} = L_{x}L_{s}$ must also hold. Therefore, 
	$s$ and $x$ operator commute too.
\end{proof}

\subsection{The easy inclusion}
Next, we prove  the inclusion ``$\subseteq$'' in \eqref{eq:g1}, when $\msub = \hat \partial$. 
\begin{proposition}[The easy inclusion] \label{prop:easy}
	Let $\spec{f}:\jAlg \to \ReIm$ be the spectral function induced 
	by a symmetric function $f:\Re^r \to \ReIm$. 
Let $s \in \hat \partial F(x)$. Then, $x$ and $s$ operator commute and for any $\jFr \in \jFr(x,s)$ we have 
\[
\diag(s,\jFr)\in \hat \partial f(\lambda(x)).
\]
\end{proposition}
\begin{proof}
	Let $s \in \hat \partial \spec{f}(x)$.
	By Proposition \ref{prop:rw} there exists a neighborhood $\mathcal{U}$ of $x$ and a $C^1$
	function $H:\mathcal{U}\to \Re$ such that $H(y) \leq F(y)$ for all $y \in \mathcal{U}$ and 
	$H(x) = F(x)$, $\nabla H(x) = s$. In addition, by Proposition~\ref{prop:p_com},
	$s$ and $x$ operator commute. Therefore,   $\jFr(x,s)$ must be nonempty, i.e.,
	$x$ and $s$ have at least one common Jordan frame. 
	
	Let $\jFr \in \jFr(x,s)$ and consider the linear map $\Diag(\cdot,\jFr):\Re^r \to\jAlg$. Since $\Diag(\cdot,\jFr)$ is continuous, $V = \Diag(\cdot,\jFr)^{-1}(\mathcal{U})$ is 
	an open set of $\Re^r$ containing $\lambda(x)$. 
	Now, let $h: V\to \Re$ be such that 
\[
	h(v)\coloneqq H(\Diag(v,\jFr)).
\]
	Let $v \in V$. Using the symmetry of $f$ and the properties of $H$, we obtain
	\begin{align*}
	f(v) &=  f(\lambda(\Diag(v,\jFr) )) \\
	&\geq H(\Diag(v,\jFr)) \\
	&= h(v).
	\end{align*}
That is, $f(v) \geq h(v)$ holds for every $v \in V$. 
Also $h(\lambda(x)) = H(\Diag(\lambda(x),\jFr)) = H(x) = f(\lambda(x))$.
By the chain rule, we also have $\nabla h(\lambda(x)) = \diag(s,\jFr)$.
Therefore, by Proposition \ref{prop:rw}, we conclude 
that $\diag(s,\jFr) \in \hat \partial f(\lambda(x))$.
%
%
%
%
\end{proof}

\subsection{The hard inclusion}\label{sec:hard}
The hard part of proving \eqref{eq:g1} is establishing 
the inclusion ``$\supseteq$'', when $\msub = \hat \partial$.
From Lewis' discussion in \cite{Le99}, it seems that one of the key steps for 
proving \eqref{eq:g1} in the case of symmetric matrices is 
a result relating the diagonal of a matrix $Z$ with the directional derivative 
$\lambda'(X;Z)$, see Theorem 5 in \cite{Le99}.
We will prove an analogous result by following an original approach 
making use of a recent {majorization} principle 
proved by Gowda in \cite{Go17}.

Let $u \in \Re^r$, we denote by $\Sd{u}$ the element in $\Re^r_{\geq}$  
corresponding to a reordering of the coordinates of $u$ in such a way 
that
\[
\Sd{u}_1 \geq \cdots \geq \Sd{u}_r.
\]
Now, let $v \in \Re^r$ be another element. Then, we say that  $u$ is  \emph{majorized} by $v$ and write $u \prec v$ if 
\[
\sum _{i=1}^k \Sd{u}_i \leq \sum _{i=1}^k \Sd{v}_i, \quad \forall k \in \{1,\ldots, r-1\}
\]
and the sum of components of both $u$ and $v$ coincide, i.e.,  $u_1+\cdots + u_r = v_1+\cdots+v_r$.
It is a classical fact following from Birkhoff's theorem that $u $ is majorized by $v$ if and only if 
$v$ lies in the convex hull of all permutations of $v$, i.e., 
\[
u \in \conv \{Pv \mid P \in \mathcal{P}^r \},
\]
see Section B in Chapter 2 of \cite{OMA16}.
%
%
If $x, y \in \jAlg$ we say that \emph{$x$ is majorized by $y$} and write 
$x \prec y$ if 
$\lambda(x)$ is majorized by $\lambda(y)$. 
Whenever majorization principles are used, it is safer
to mention the standard disclaimers that, throughout the literature, there seems to be no consensus on the direction of the inequalities appearing in the definition of majorization. In some texts, ``$\geq$'' is used instead of ``$\leq$''. Here, we are following the convention in \cite{Go17}, which by its turn follows the notation in \cite{Bh97}.


Let $X$ be a symmetric matrix. It is known that the diagonal entries of 
$X$ are majorized by the eigenvalues of $X$.
Gowda recently extended this fact to Euclidean Jordan algebras.
\begin{proposition}[Gowda, Example 7 and Theorem 6 in \cite{Go17}]\label{prop:maj}
Let $\jFr$ be a 
Jordan frame and let $x \in \jAlg$.
Then, $\diag(x, \jFr)$ is majorized by $\lambda(x)$. 
In particular,
\[
\diag(x,\jFr) \in \conv \{P\lambda(x) \mid P \in \Pe{r} \}.
\]
\end{proposition}
\begin{proof}
Consider the map $\psi:\jAlg \to \jAlg$ defined by 
\[
\psi(x) \coloneqq \sum _{i=1}^r \inProd{c_i}{x} c_i, \qquad \forall x \in \jAlg.
\]
In \cite{Go17}, the map $\psi$ is denoted by ``$\Diag$'' and it has 
a different meaning from the map $\Diag$ we are using in this paper. 
In any case, in Example 7 and Theorem 6 in \cite{Go17}, Gowda showed 
that $\psi(x) \prec x$ holds for every $x \in \jAlg$.
Accordingly, we have
\[
\lambda(\psi(x) ) \prec \lambda(x).
\]
Now, we observe that the components of $\diag(x,\jFr)$ are precisely the eigenvalues of 
$\psi(x)$. Furthermore, 
the fact that a vector $u \in \Re^r$ is majorized by $v \in \Re^r$ does not change if 
we permute the entries of  $u$ or $v$. We 
conclude that $
\diag(x,\jFr) \prec \lambda(x)$ and that 
$
\diag(x,\jFr) \in \conv \{P\lambda(x) \mid P \in \Pe{r} \}.
$
\end{proof}
We are now able to prove an analogous of Theorem 4 of \cite{Le99} for Euclidean Jordan algebras.
\begin{theorem}[The diagonal map and directional derivatives of the eigenvalue map]\label{theo:diag_eig}
	Let $x,z \in \jAlg$ and let $\jFr \in \jFr(x)$. Then 
\[
	\diag(z, \jFr) \in \conv \{ P\lambda'(x;z) \mid P \in \Pe{r}(\lambda(x))  \}
\]
\end{theorem}
First, we sketch the general proof strategy for Theorem~\ref{theo:diag_eig}. 
The idea is to separate the vector $\lambda(x)$ in blocks 
of equal eigenvalues and apply the formula in Theorem \ref{theo:dir} for each block.
Then, for each block, we associate a Euclidean Jordan algebra $\jAlg^j$ and 
invoke Proposition \ref{prop:maj}. Since Proposition \ref{prop:maj} is invoked in a blockwise fashion according to the blocks of equal eigenvalues of $x$, the resulting pieces can be glued together to obtain a convex combination of matrices in $\Pe{r}(\lambda(x))$.
\begin{proof}
To start, let us consider the spectral decomposition of $x$,
\[
x = \sum _{i=1}^r \lambda_i(x) c_i,
\]
where $\lambda _1(x) \geq \cdots \geq \lambda _r(x)$ and 
$\jFr = [c_1,\ldots, c_r]$ is a Jordan frame.
Now, we use the notation described in Section~\ref{sec:eig_dir} and denote 
by $l_i(x)$ the ``relative position'' of the index $i$ with respect 
the eigenvalues of $x$ that are equal to $\lambda _i(x)$.

Next, let $r_1,\ldots, r_\ell$ be such that 
\begin{multline*}
	\lambda _1(x) = \cdots = \lambda _{r_1}(x) > \lambda _{r_1+1}(x) = \cdots =
	\lambda _{r_2}(x) >\\ \lambda_{r_2+1}(x) =  \cdots = \lambda _{r_3}(x) >\cdots \lambda _{r_\ell}(x).
\end{multline*}
Here, $\ell$  is the number of \emph{distinct} eigenvalues of $x$. 
For convenience, we define $r_0 = 0$ and $n_j = r_{j} - r_{j-1}$ for 
$j \in \{1,\ldots, \ell\}$.
Then, we  divide $\diag(z,\jFr)$ in $\ell$ parts according to the blocks of equal eigenvalues of $x$:
\[
\diag(z,\jFr) = (u^1,\ldots, u^\ell)
\]
where 
\[
u^j  = (\inProd{z}{c_{r_{j-1}+1}},\ldots, \inProd{z}{c_{r_j}} ) \in \Re^{n_j}.
\]
We do the same for the map $\lambda$ and divide $\lambda$ in 
$\ell$ maps such that 
\begin{equation*}
\lambda(y) = (\lambda^1(y),\ldots, \lambda^\ell(y)), \quad \forall y \in \jAlg.
\end{equation*}
Here, each $\lambda^j$ is a map $\jAlg \to \Re^{n_j}$ such that 
\begin{equation*}
\lambda^j(y)  \coloneqq ({\lambda_{r_{j-1}+1}}(y),\ldots, {\lambda_{r_{j}}}(y) ) \in \Re^{n_j}.
\end{equation*}
Applying Theorem \ref{theo:dir} to each $\lambda^j$, we obtain
\begin{equation}\label{eq:eig_dir_aux}
(\lambda^j)'(x;z) = (\lambda _{l_{r_{j-1}+1}}(Q_{e_{r_j}}(z),\jAlg^j), \ldots,  \lambda _{l_{r_{j}}}(Q_{e_{r_j}}(z),\jAlg^j) ),
\end{equation}
where $e_{r_j}$ is the sum of the idempotents associated to 
the eigenvalues equal to $\lambda _{r_j}(x)$ and  $\jAlg^j$ is 
the Jordan algebra $V(e_{r_j},1)$ of rank $n_j$.

Let $z^j \coloneqq Q_{e_{r_j}}(z)$, for every $j$.
From Theorem~\ref{theo:peirce1}, $z^j$ is the orthogonal projection of $z$ onto $\jAlg^j$. The indices from $r_{j-1}+1$ to $r_{j}$ all correspond to equal eigenvalues of $x$. Therefore, from \eqref{eq:eig_dir_aux} and
the definition of the relative index $l_{r_{j-1}+k}$,  
 we conclude that  
\begin{equation}\label{eq:eig_aux}
(\lambda^j)'(x;z) =  (\lambda _{1}(z^j,\jAlg^j), \ldots,  \lambda _{n_j}(z^j,\jAlg^j) ) = \lambda(z^j, \jAlg^j),
\end{equation}
where we recall that $\lambda(\cdot,\jAlg^j)$ is the eigenvalue 
map of the algebra $\jAlg^j$.
Next, let $\jFr^j \coloneqq [c_{r_{j-1}+1},\ldots, c_{r_j}]$. Since 
$\jFr$ is a Jordan frame and the sum of the elements of $\jFr^j$ is $e_{r_j}$ (the identity element of $\jAlg^j$), we have that $\jFr^j$ is a Jordan frame in the algebra 
$\jAlg^j$. 
We will now prove 
that $\diag(z^j,\jFr^j) = u^j$. 
Let $k$ be an integer such that $r_{j-1}+1 \leq k \leq {r_j}$, 
we have
\begin{align*}
\inProd{z^j}{c_k} & = \inProd{Q_{e_{r_j}}z}{c_k} \\
& = \inProd{z}{Q_{e_{r_j}}c_k} \\
& = \inProd{z}{c_k},
\end{align*}
where the second equality follows from the fact $Q_{e_{r_j}}$ is self-adjoint and 
the third equality follows from the fact that $Q_{e_{r_j}}(c_k) = c_k$ since $e_{r_j}$ is the identity 
element in $\jAlg^j$ and $c_k$ is an idempotent contained in $\jAlg^j$.
Since this holds for every $k$ satisfying $r_{j-1}+1 \leq k \leq {r_j}$, 
we conclude that $\diag(z^j,\jFr^j) = u^j$.
From \eqref{eq:eig_aux} and Proposition~\ref{prop:maj} applied to $z^j,\jFr^j$ and $\jAlg^j$, we conclude that for 
every $j$, we have 
\[
\diag(z^j,\jFr^j) = u^j \in \conv \{P((\lambda^j)'(x;z)) \mid P \in \Pe{n_j} \}.
\]
That is, there are nonnegative coefficients $\alpha _{j,k} $ and 
$\kappa _j$ permutation matrices $P^{j,k} \in \Pe{n_{j}}$ such that 
\begin{equation}\label{eq:uj}
u^j = \sum _{k=1}^{\kappa_j} \alpha _{j,k}P^{j,k}((\lambda^j)'(x;z)) \qquad 
\sum _{k=1}^{\kappa _j} \alpha _{j,k} = 1.
\end{equation}
We are now almost done.
First, we define $A^j$ as the following $n_j \times n_j$ matrix
\begin{equation}\label{eq:Aj}
A^j \coloneqq \sum _{k=1}^{\kappa_j} \alpha _{j,k}P^{j,k}.
\end{equation}
Next, we define $A$ as the matrix satisfying 
\begin{equation}\label{eq:A}
A = \sum _{j_1=1}^{\kappa_1}\sum _{j_2=1}^{\kappa_2} \cdots 
\sum _{j_{\ell}=1}^{\kappa_{\ell}}\alpha_{1,j_1}\cdots \alpha _{\ell,j_{\ell}} \begin{pmatrix}
P^{1,j_1} & &\\
& \ddots   & \\
& & P^{\ell,j_{\ell}}
\end{pmatrix}.
\end{equation}
Because of \eqref{eq:Aj}, we have
\[
A = \begin{pmatrix}
A^1 & &\\
& \ddots   & \\
& & A^\ell 
\end{pmatrix},
\]
which together with \eqref{eq:uj} implies that 
\begin{equation}\label{eq:As}
\diag(z,\jFr) = (u^1,\ldots, u^\ell) = A\lambda'(x;z).
\end{equation}
Now, we consider an arbitrary matrix  $P$ appearing in \eqref{eq:A} which is of the form
\[
P = \begin{pmatrix}
P^{1,j_1} & &\\
& \ddots   & \\
& & P^{\ell,j_{\ell}}
\end{pmatrix}.
\]
$P$ is a block diagonal matrix and since each block is a permutation matrix, $P$ is 
a permutation matrix too. Furthermore, by construction, the block structure of $P$ follows the pattern of equal eigenvalues of $x$. So, for instance, 
$P^{1,j_1}$ has size $n_1 = r_1$, which corresponds to the first block of $r_1$ equal eigenvalues of $x$. For this reason, we obtain
\[
P\lambda(x) = (P^{1,j_1}\lambda^1(x),\ldots, P^{\ell,j_{\ell}}\lambda^\ell (x) ) = 
\lambda(x).
\]
Accordingly, $P$ belongs to $\Pe{r}(\lambda(x))$ and from \eqref{eq:A} and \eqref{eq:As}, we conclude that
\[
\diag(z, \jFr) \in \conv \{ P\lambda'(x;z) \mid P \in \Pe{r}(\lambda(x))  \}.
\]
\end{proof}
Next, we will prove the inclusion ``$\supseteq$'' in \eqref{eq:g1}, when $\msub = \hat \partial$. With all the preliminary results in place, we can proceed analogously to Theorem~5 of \cite{Le99}.
\begin{proposition}[The hard inclusion]\label{prop:hard}
	Let $\spec{f}:\jAlg \to \ReIm$ be the spectral function induced 
	by a symmetric function $f:\Re^r \to \ReIm$. Then
\[
\hat \partial F(x) \supseteq \{ s \in \jAlg \mid \exists \jFr \in \jFr(x,s) \text { with } \diag(s,\jFr) \in \hat \partial f(\lambda(x))
	\}.
\]
\end{proposition}
\begin{proof}
Let $s \in \jAlg$ and $\jFr \in \jFr(x,s)$ be such that  $\diag(s,\jFr) \in \hat \partial f(\lambda(x))$. Our goal is to show that $s \in \hat  \partial F(x).$ In view of \eqref{eq:reg_sub2},  $s \in \hat  \partial F(x)$ will be 
established if we show that for every $\epsilon > 0$, there exists $\delta$ such that $\norm{z} \leq \delta$ implies 
\[
f(\lambda(x+z)) \geq f(\lambda(x)) + \inProd{s}{z} - \epsilon \norm{z}.
\]
However, since $\jFr$ diagonalizes $s$, we have 
\[
\inProd{s}{z} = \inProd{\sum _{i=1}^r \inProd{s}{c_i}c_i}{z} = {\sum _{i=1}^r \inProd{s}{c_i}\inProd{c_i}{z}} = \inProd{\diag(s,\jFr)}{\diag(z,\jFr)}. 
\]
Therefore, our goal is to show that for every $\epsilon > 0$, there exists $\delta$ such that $\norm{z} \leq \delta$ implies 
\begin{equation}
f(\lambda(x+z)) \geq f(\lambda(x)) + \inProd{\diag(s,\jFr)}{\diag(z,\jFr)} - \epsilon \norm{z}. \tag{Goal} \label{eq:goal}
\end{equation}
Now, we will set up a few objects that will help us towards proving \eqref{eq:goal}.
First, we observe that $\diag(s,\jFr) \in \hat \partial f(\lambda(x))$ and 
\eqref{eq:sym_sub} implies that 
\[
P\diag(s,\jFr) \in \hat \partial f(\lambda(x)),\quad \forall P \in \Pe{r}(\lambda(x)).
\]
Next, we define $\Lambda$ to be the convex hull of the $P\diag(s,\jFr)$ with $ P \in \Pe{r}(\lambda(x))$ and denote by $\delta^*_{\Lambda}$ the corresponding support 
function. Since $\Lambda$ is generated by a finite number of elements, we have 
\[
\delta^*_{\Lambda}(v) = \sup\{\inProd{v}{\hat v} \mid \hat v \in \Lambda \} =  \max \{\inProd{P\diag(s,\jFr)}{v} \mid P \in \Pe{r}(\lambda(x)) \}.
\]
Now that the pieces are in place, we move on to proving \eqref{eq:goal}.
Let $\epsilon > 0$. From the definition of regular subgradients (see \eqref{eq:reg_sub}) and from \eqref{eq:reg_sub2}, for every $ P \in \Pe{r}(\lambda(x))$, there exists $\delta_P$ such that $\norm{v}\leq \delta _P$ implies 
\[
f(\lambda(x)+v) \geq f(\lambda(x)) + \inProd{P\diag(s,\jFr)}{v} -\epsilon\norm{v}.
\]
In particular, if we let $\hat \delta = \min _{P  \in \Pe{r}(\lambda(x))}\delta_P$, we 
conclude that
\begin{align}
f(\lambda(x)+v) &\geq \max \{ f(\lambda(x)) + \inProd{P\diag(s,\jFr)}{v} -\epsilon\norm{v}\mid P  \in \Pe{r}(\lambda(x))  \} \notag\\
&= f(\lambda(x)) +\delta^*_{\Lambda}(v) -\epsilon\norm{v}, \label{eq:in_sup}
\end{align}
whenever $\norm{v} \leq \hat \delta$.
From item $(ii)$ of Lemma \ref{lem:eig} and decreasing $\hat \delta$ if necessary, we have  that if 
$z \in \jAlg$ satisfies $\norm{z} \leq \hat \delta$, it holds that
\begin{equation}\label{eq:eig_dir}
 \norm{\lambda(x+z)-\lambda(x) - \lambda'(x;z) } \leq \epsilon{\norm{z}}.
\end{equation}
By item $(i)$ of Lemma \ref{lem:eig}, $\norm{\lambda(x + z)-\lambda(x)} \leq \norm{z}$. Therefore, if $z$ satisfies $\norm{z} \leq \hat \delta$, 
we obtain from \eqref{eq:in_sup} that
\begin{align}
f(\lambda(x + z)) & = f(\lambda(x) + (\lambda(x + z)-\lambda(x))) \notag\\
& \geq f(\lambda(x)) - \epsilon 
\norm{z} + \delta^*_{\Lambda}(\lambda(x + z)-\lambda(x)).  \label{eq:aux}
\end{align}
Since $\delta^*_{\Lambda}$ is the pointwise maximum of linear functions, 
$\delta^*_{\Lambda}$ is a Lipschitz continuous  sublinear function with Lipschitz constant $\kappa$  given by
\[
\kappa = \max _{P \in \Pe{r}(\lambda(x))} \norm{P\diag(s,\jFr)} = \norm{\diag(s,\jFr)}.
\]
Therefore,  for every $u,v \in \Re^r$, we have 
\begin{align}
\delta^*_{\Lambda}(u+v) & \geq \delta^*_{\Lambda}(u) - \delta^*_{\Lambda}(-v) \notag \\
&\geq \delta^*_{\Lambda}(u) - \kappa\norm{v}. \label{eq:aux2}
\end{align} 
Now, we let $u = \lambda'(x;z)$ and $v = \lambda(x+z) - \lambda(x) - \lambda'(x;z)$ in 
\eqref{eq:aux2} and use the resulting inequality back in \eqref{eq:aux}, to obtain 
\begin{align}
f(\lambda(x + z)) & \geq f(\lambda(x)) + \delta^*_{\Lambda}(\lambda'(x;z)) - \epsilon \norm{z} - \kappa\norm{\lambda(x+z) - \lambda(x) - \lambda'(x;z)}\notag \\
& \geq f(\lambda(x)) + \delta^*_{\Lambda}(\lambda'(x;z)) - (1+\kappa)\epsilon\norm{z}
\label{eq:aux3},
\end{align}
where the last inequality follows from \eqref{eq:eig_dir}.

By Theorem \ref{theo:diag_eig}, we have
\[
\diag(z,\jFr) \in \conv\{P \lambda'(x;z) \mid P \in \Pe{r}(\lambda(x)) \}.
\]
Therefore, there are nonnegative numbers $\alpha_1,\ldots,\alpha_{\ell}$ such that their sum is $1$ and 
\[
\diag(z,\jFr) = \sum _{i=1}^\ell \alpha_i P_i \lambda'(x;z),
\]
where each $P_i$ belongs to $\Pe{r}(\lambda(x))$. We recall that, by definition, $\delta^*_{\Lambda}(Pu) = \delta^*_{\Lambda}(u) $ for every $P\in \Pe{r}(\lambda(x))$ and $u \in \Re^r$.
Using the convexity of $\delta^*_{\Lambda}$, we obtain
\begin{align}
\delta^*_{\Lambda}(\diag(z,\jFr)) & \leq \sum _{i=1}^\ell \alpha_i \delta^*_{\Lambda}(P_i \lambda'(x;z)) \notag \\
& =  \sum _{i=1}^\ell \alpha_i \delta^*_{\Lambda}(\lambda'(x;z)) \notag \\
& = \delta^*_{\Lambda}(\lambda'(x;z)), \label{eq:aux4}
\end{align}
Using inequality \eqref{eq:aux4} in \eqref{eq:aux3}, we obtain that for every 
$z \in \jAlg$ with $\norm{z} \leq \hat \delta$, we have 
\begin{align}
f(\lambda(x + z)) & \geq  f(\lambda(x)) + \delta^*_{\Lambda}(\lambda'(x;z)) - (1+\kappa)\epsilon\norm{z} \notag \\
&\geq f(\lambda(x)) + \delta^*_{\Lambda}(\diag(z,\jFr)) - (1+\kappa)\epsilon\norm{z} \notag\\
&\geq f(\lambda(x)) + \inProd{\diag(s,\jFr)}{ \diag(z,\jFr)} - (1+\kappa)\epsilon\norm{z}. \notag
\end{align}
Since $\epsilon$ was arbitrary, this shows that \eqref{eq:goal} holds. 
\end{proof}

%
%

%

\subsection{Main results}
From Propositions~\ref{prop:easy} and \ref{prop:hard}, we conclude
that \eqref{eq:g1} holds for the case $\msub = \hat \partial$. Next, will prove transfer results for the approximate and horizon subdifferentials which will conclude the proof of \eqref{eq:g1}.

%
%

\begin{proposition}[The approximate and horizon subdifferentials of spectral functions]\label{prop:sub_app}
	Let $F:\jAlg \to \ReIm$ be the spectral function induced by a symmetric function $f: \Re^r \to \ReIm$. Then, for $x \in \jAlg$, we have
	\begin{align}
	\partial F(x) &= \{ s \in \jAlg \mid \exists \jFr \in \jFr(x,s) \text { with } \diag(s,\jFr) \in \partial f(\lambda(x))
	\}. {\label{eq:g_app}}\\
	\hsub F(x) &= \{ s \in \jAlg \mid \exists \jFr \in \jFr(x,s) \text { with } \diag(s,\jFr) \in \hsub f(\lambda(x))
	\}.\label{eq:g3}
	\end{align}	
\end{proposition}
\begin{proof}
First, we prove the inclusion ``$\subseteq$'' in \eqref{eq:g_app} and \eqref{eq:g3}. 
Let $s \in \partial F(x)$ or $s \in \hsub F(x)$. 
By definition, there are sequences $\{x^k\},\{s^k\},\{t^k\}$ such that $s^k \in \hat \partial F(x^k)$ holds for every $k$ and 
\[
x^k \to x,\quad f(\lambda(x^k))\to f(\lambda(x)),\quad  t^k s^k \to s.
\]
Here, there are two cases for $\{t^k\}$. If 
$s \in \partial \spec{f}(x)$, then $t^k = 1$ for every $k$. If $s \in \hsub \spec{f}(x)$, then $t^k \downarrow 0$.
Since $s^k \in \hat \partial F(x^k)$ holds for every $k$, Proposition~\ref{prop:easy} implies the existence of $\jFr^k \in \jFr(x^k,s^k)$ such that 
\[
\diag(s^k,\jFr^k) \in \hat \partial f(\lambda(x^k)), \,\, \forall k.
\]
Let $\jFr^k = [c_{1,k},\ldots, c_{r,k}]$. Since $\norm{c_{i,k}}=1$ for every $i$ and $k$, passing to a subsequence if necessary, we may assume that for every $i$, $c_{i,k}$ converges to some $\overline{c}_i$. Elementary properties of limits show that 
$\jProd{\overline{c}_i}{\overline{c}_j} = 0$ if $i\neq j$ and $\jProd{\overline{c}_i}{\overline{c}_i} = \overline{c}_i$. 
Therefore $\overline{\jFr} = [\overline{c}_1,\ldots, \overline{c}_r ]$ is a Jordan frame in $\jAlg$.

Now, we need to examine whether $\overline{\jFr} \in \jFr(x,s)$. We have
\[
x^k = \sum _{i=1}^r \lambda_{i}(x^k)c_{i,k}.
\]
Since each $\lambda _i(\cdot)$ is a continuous function and $x^k \to x$, we conclude that 
\[
x =  \sum _{i=1}^r \lambda_{i}(x)\overline{c}_{i}.
\]
An analogous argument shows that $\overline{\jFr}$ diagonalizes $s$. 
Gathering all we have shown, we obtain that $\diag(s^k,\jFr^k) \in \hat \partial f(\lambda(x^k)) $ holds for every $k$ and 
\[
\lambda(x^k) \to \lambda(x),\quad  f(\lambda(x^k)) \to f(\lambda(x)), \quad t^k\diag(s^k,\jFr^k) \to \diag(s,\overline{\jFr}).
\]
That is, $\overline{\jFr} \in \jFr(x,s)$ together with either
$\diag(s,\overline{\jFr}) \in \partial f(\lambda(x))$ (if $s \in \partial F(x)$) or 
 $\diag(s,\overline{\jFr}) \in \hsub f(\lambda(x))$ (if $s \in \hsub F(x)$).
 
We will now prove the inclusion ``$\supseteq$''.
Let $s \in \jAlg$ be such that 
there are sequences $\{u^k\},\{d^k\},\{t^k\}$  satisfying  $d^k \in \hat \partial f(u^k)$ for every $k$ and 
\[
u^k \to \lambda(x),\quad f(u^k) \to f(\lambda(x)),\quad t^k d^k \to \diag(s,\jFr),
\]
where $\jFr \in \jFr(x,s)$. Here, either $t^k = 1$ for every $k$ or $t^k \downarrow 0$.
 Let $\jFr = [c_1,\ldots, c_r]$.

For every $k$, let $P^k \in \Pe{r}$ be a permutation matrix such that 
$P^ku^k = \Sd{(u^k)}$. Since $ d^k \in \hat \partial f(u^k)$ holds for every $k$ and $f$ is a symmetric function, we have from \eqref{eq:sym_sub} that
\begin{equation}
P^kd^k \in \hat \partial f(\Sd{(u^k)}), \quad \forall k. \label{eq:ap_aux}
\end{equation}
Let
\[
x^k \coloneqq \Diag(u^k,\jFr),\quad s^k \coloneqq \Diag(d^k,\jFr), \quad \forall k.
\]
Let $\sigma$ be the permutation on the set $\{1,\ldots, r\}$ induced 
by $P^k$, i.e., $\sigma(i) = j$, if and only if, $P^k$ permutes the $i$-th and the $j$-th entries of a vector.
We have $\lambda(x^k) = \Sd{(u^k)}$ and $P^k\jFr \in \jFr(x^k,s^k)$, 
where $P^k\jFr $ is defined as
\[
P^k\jFr \coloneqq [c_{\sigma^{-1}(1)}, \ldots, c_{\sigma^{-1}(r)}].
\]
Therefore, from \eqref{eq:ap_aux} we have \begin{equation}
\diag(s^k,P^k\jFr) = P^kd^k \in \hat \partial f(\lambda(x^k)), \notag
\end{equation}
which combined with Proposition~\ref{prop:hard} shows that
\[
s^k \in \hat \partial F(x^k), \quad \forall k.
\]
Next, since $u^k \to \lambda(x)$, it follows that $x^k \to x$.
Again, recalling that $f$ is a symmetric function and that 
\[
F(x^k) = f(\lambda(x^k)) = f(\Sd{(u^k)}) = f(u^k),
\]
we have $F(x^k) \to F(x)$, since $f(u^k) \to f(\lambda(x))$.
Similarly, we have $t^k s^k \to s$, since $t^k d^k \to \diag(s,\jFr)$.
This shows that $s \in \partial F(x)$ (if $\diag(s,\jFr) \in \partial f(\lambda(x))$) or $s \in \hsub F(x)$ (if $\diag(s,\jFr) \in \hsub f(\lambda(x))$).
\end{proof} 
We can now state our main result.
%
%
%
%
%
%
\begin{theorem}[Generalized subdifferentials of spectral functions]\label{theo:main}
	Let $(\jAlg, \jProd{}{})$ be a Euclidean Jordan algebra of rank $r$ and let $F:\jAlg \to \ReIm$ be the spectral function induced by a symmetric function $f: \Re^r \to \ReIm$. Then, for $x \in \jAlg$, we have
	\begin{align}
	\msub F(x) &= \{ s \in \jAlg \mid \exists \jFr \in \jFr(x,s) \text{ with } \diag(s,\jFr) \in \msub f(\lambda(x))\}, {\label{eq:g_gen}} \tag{Transfer}
	\end{align}
	whenever  $\msub$ is $\hat \partial, \partial$ or $\hsub$.	
	
\end{theorem}
\begin{proof}
Follows from Propositions~\ref{prop:easy}, \ref{prop:hard}, \ref{prop:sub_app}.
%
%
\end{proof}

\subsection{Convex hull of generalized subdifferentials and the Clarke subdifferential}\label{sec:conv}
In this subsection, we will  prove the following meta-formula
\[
\conv \msub F(x) = 
\{ s \in \jAlg \mid \exists \jFr \in \jFr(x,s) \text { with } \diag(s,\jFr) \in \conv \msub f(\lambda(x))
\},
\]
whenever $\msub$ is a subdifferential which behaves nicely with respect to permutations and for which \eqref{eq:g_gen} holds.
One of the motivations for this formula is, of course, the study 
of the Clarke subdifferential, which we will discuss next.
First, we recall that $f$ is \emph{locally Lipschitz continuous 
at $\hat {u}$} if there exists some neighbourhood $U$ of $\hat {u}$ and a constant $\kappa$ such that 
\[
\abs{f(v) - f(u)} \leq \kappa \norm{v - u}, \quad \forall u,v \in U\cap \dom f.
\]

Using the construction of the Clarke subdifferential through 
the Bouligand derivative, Baes proved in his PhD thesis that, if $f$ is locally Lipschitz, 
then the meta-formula \eqref{eq:g_gen} holds 
when $\msub$ is either the Bouligand or the Clarke subdifferential, see 
Proposition 4.5.1 and Theorems~4.5.4 and 4.5.5 in \cite{B06phd}.
However, denoting by $\csub$ the Clarke subdifferential, it turns out that, under local Lipschitzness, we have
\[
\csub f(u) = \conv \partial f(u), \quad \forall u \in \interior (\dom f),
\]
 see Theorem~9.61 in \cite{RW}.
Therefore, with some effort, Theorem~\ref{theo:main} can 
be used to give another proof that \eqref{eq:g_gen} holds when 
$\msub$ is $\csub$ and $f$ is locally Lipschitz continuous. 
The first step towards this idea is the following result,
which is a variant of Theorem~\ref{theo:diag_eig}.

\begin{proposition}\label{prop:diag_eig}
Let $x,s\in \jAlg$ be such that $x$ and $s$ operator commute. Then, for every $\jFr \in \jFr(x)$ and 
every $\hat \jFr \in \jFr(x,s)$ we have
\[
\diag(s,\jFr) \in \conv \{ P \diag(s, \hat\jFr) \mid P \in \Pe{r}(\lambda(x)) \}.
\]
\end{proposition}
\begin{proof}
By Theorem~\ref{theo:diag_eig}, we already have
\begin{equation}
\diag(s, \jFr) \in \conv \{ P \lambda'(x,s) \mid P \in \Pe{r}(\lambda(x)) \}. \label{eq:conv_aux}
\end{equation}
All we need to do now is to relate $\lambda'(x,s)$ and $\diag(s, \hat\jFr)$.
For that, we will proceed as in the proof of Theorem~\ref{theo:diag_eig}.

Let us consider the spectral decomposition of $x$ according 
to $\hat \jFr = [\hat c_1, \ldots, \hat c_r]$,
\[
x = \sum _{i=1}^r \lambda_i(x) \hat c_i.
\]
Then, we use the notation described in Section~\ref{sec:eig_dir} and denote 
by $l_i(x)$ the ``relative position'' of the index $i$ with respect 
the eigenvalues of $x$ that are equal to $\lambda _i(x)$.
Furthermore, we let $e_i$ be the sum of the idempotents $\hat c_i$ associated to 
the eigenvalues equal to $\lambda _i(x)$. We also let $r_1,\ldots, r_\ell$ be such that 
\begin{multline*}
\lambda _1(x) = \cdots = \lambda _{r_1}(x) > \lambda _{r_1+1}(x) = \cdots =
\lambda _{r_2}(x) >\\ \lambda_{r_2+1}(x) =  \cdots = \lambda _{r_3}(x) >\cdots \lambda _{r_\ell}(x).
\end{multline*}
Here, $\ell$  is the number of \emph{distinct} eigenvalues of $x$. 
For convenience, we define $r_0 = 0$ and $n_j = r_{j} - r_{j-1}$ for 
$j \in \{1,\ldots, \ell\}$.
Then, we  divide $\diag(s,\hat \jFr)$ and $\lambda'(x;s)$ in $\ell$ parts according to the blocks of equal eigenvalues of $x$:
\begin{align*}
\diag(s,\hat \jFr) &= (u^1,\ldots, u^\ell)\\
\lambda'(x;s) &= (v^1,\ldots, v^\ell).
\end{align*}
First, we observe that if $\lambda_i(x) = \lambda_j(x)$, then we have $e_i = e_j$. Then, from the formula for the directional derivatives (Theorem~\ref{theo:dir}) and the fact that  $\hat \jFr$ diagonalizes $s$, we obtain 
\begin{align*}
u^j & = (\inProd{s}{\hat c_{r_{j-1}+1}},\ldots, \inProd{s}{\hat c_{r_j}} ) \in \Re^{n_j}\\
v^j & = (\lambda _{l_{r_{j-1}+1}}(Q_{e_{r_j}}(s);\jAlg^j), \ldots,  \lambda _{l_{r_{j}}}(Q_{e_{r_j}}(s);\jAlg^j) ) \in \Re^{n_j},
\end{align*}
where $\jAlg^j  = V(e_{r_j},1)$.
We recall that $Q_{e_{r_j}}(s)$ is the orthogonal projection of $s$ onto 
$V(e_{r_j},1)$. And, again, because $\hat \jFr$ diagonalizes $s$, we 
obtain 
\[
Q_{e_{r_j}}(s) = \sum _{i=r_{j-1}+1}^{r_j} \inProd{s}{\hat c_{i}} \hat c_{i},
\]
which is the spectral decomposition of $Q_{e_{r_j}}(s)$ in the algebra
$\jAlg^j$. In particular, the eigenvalues of $Q_{e_{r_j}}(s) $ in the algebra $\jAlg^j$ are precisely the components of $u^j$. 
We also need to recall that $\lambda _{l_{r_{j-1}+k}}(Q_{e_{r_j}}(s);\jAlg^j)$ is, in fact, the $k$-th largest 
eigenvalue of $Q_{e_{r_j}}(s)$ in the algebra $\jAlg^j$.

Piecing everything together, we conclude that $v^j$ is just the result of 
sorting $u^j$ in nonincreasing order.  Therefore, there exists a permutation matrix $P^j \in \Pe{n_j}$ such that $v^j = P^ju^j$, for every $j \in \{1,\ldots, \ell \}$. Then, if we let 
\[
\hat P = \begin{pmatrix}
P^{1} & &\\
& \ddots   & \\
& & P^{\ell}
\end{pmatrix},
\]
we have $\lambda'(x,s) = \hat P\diag(s,\hat \jFr)$ and since the block structure of $P$ follows the blocks of equal eigenvalues of $\lambda(x)$, we have 
$P \in \Pe{r}(\lambda(x))$. From \eqref{eq:conv_aux},
we have
\begin{multline*}
\diag(s, \jFr) \in \conv \{ P\hat P \diag(s,\hat \jFr) \mid P \in \Pe{r}(\lambda(x)) \} = \\ \conv \{ P \diag(s,\hat \jFr) \mid P \in \Pe{r}(\lambda(x)) \},
\end{multline*}
since $\Pe{r}(\lambda(x))$ is a group.
\end{proof}
For what follows, we say that a subdifferential $\msub$ is \emph{permutation compatible} if 
\[
\msub f(Pu) = P \msub f(u),\quad \forall u \in \Re^r
\]
whenever $f:\Re^r \to \ReIm$ is a symmetric function and $P \in \Pe{r}$. 
We note that all subdifferentials $\hat \partial,\partial, \hsub, \csub$ that have appeared so far in this paper are 
permutation compatible.
With that, we are ready to prove the following meta-theorem which might be applicable to other subdifferentials not discussed in this paper. 
\begin{theorem}[Convex hull of generalized subdifferentials]\label{theo:hull}
	Let $F:\jAlg \to \ReIm$ be the spectral function induced by a symmetric function $f: \Re^r \to \ReIm$. Then, for $x \in \jAlg$, we have
\begin{equation}
\conv \msub F(x) =  
\{ s \in \jAlg \mid \exists \jFr \in \jFr(x,s) \text { with } \diag(s,\jFr) \in \conv \msub f(\lambda(x))
\}, {\label{eq:g_conv}} \tag{Transfer-Hull}
\end{equation}
where $\msub$ is any permutation compatible subdifferential for which \eqref{eq:g_gen} holds.
In particular, if $\lambda(x) \in \interior (\dom f)$ and $f$ is locally Lipschitz continuous at $\lambda(x)$, then 
\eqref{eq:g_gen} holds when $\msub = \csub$.
\end{theorem}

\begin{proof}
First we prove the ``$\supseteq$'' inclusion.
Suppose $s$ and $\jFr$ are such that $\jFr \in \jFr(x,s)$ and  $\diag(s,\jFr)$ is the convex 
combination of $d_1,\ldots, d_\ell \in \msub  f(\lambda(x))$.
Then, since \eqref{eq:g_gen} holds, we have 
\[
\Diag(d_i,\jFr) \in \msub F(x),\quad \forall i \in \{1,\ldots,\ell\}.
\]
Because $s$ is a convex combination of the 
 $\Diag(d_i,\jFr)$, we obtain $s \in  \conv \msub F(x)$.
 
Next, we prove the ``$\subseteq$'' inclusion. Let $s_1,s_2 \in \msub F(x)$. Since 
\eqref{eq:g_gen} holds, there are $\jFr_1 \in \jFr(x,s_1)$ and $\jFr_2 \in \jFr_2(x,s_2)$ such that 
\begin{equation}
\diag(s_1,\jFr_1) \in \msub f(\lambda(x)),\quad \diag(s_2,\jFr_2) \in \msub f(\lambda(x)). \label{eq:conv_aux2}
\end{equation}
Let $s_3$ be a convex combination of $s_1,s_2$, so that
\[
s_3 = \alpha s_1 + (1-\alpha)s_2,
\]
for some $\alpha \in [0,1]$.
Since $x,s_1$ and $x,s_2$ are pairs of simultaneously diagonalizable elements, the same must be true of the pair $x,s_3$, see \eqref{eq:simultaneous}. We conclude that there exists $\jFr_3 \in \jFr(x,s_3)$.
Now, we invoke Proposition~\ref{prop:diag_eig} with $\jFr = \jFr_3$ and $\hat \jFr = \jFr_1$, to conclude that 
\[
\diag(s_1,\jFr_3) \in \conv \{P \diag(s_1,\jFr_1) \mid P \in \Pe{r}(\lambda(x)) \}.
\]
Because $\msub$ is permutation compatible, \eqref{eq:conv_aux2} implies that $P \diag(s_1,\jFr_1)$ belongs to
$\msub f(\lambda(x))$ for every $ P \in \Pe{r}({\lambda(x)})$.
Therefore, $\diag(s_1,\jFr_3) \in \conv \msub f(\lambda(x))$. A 
completely analogous argument for $s_2$ shows that  
\[
\diag(s_1,\jFr_3) \in \conv \msub f(\lambda(x)),\quad \diag(s_2,\jFr_3) \in \conv \msub f(\lambda(x)).
\]
Since $\diag(s_3,\jFr_3)$ is a convex combination of $\diag(s_1,\jFr_3) $ and $\diag(s_2,\jFr_3) $, we conclude that, indeed, 
\[
\diag(s_3,\jFr_3) \in \conv \msub f(\lambda(x)),
\]
which proves the inclusion ``$\subseteq$''.

Finally, if $f$ is locally Lipschitz continuous at $\lambda(x) \in \interior (\dom f)$, the fact that the eigenvalue map is Lipschitz continuous (Lemma~\ref{lem:eig}) shows that $F$ must be locally Lipschitz continuous at $x$. Therefore,
\[
\csub F(x) = \conv \partial F(x),\quad  \csub f(\lambda(x)) = \conv \partial f(\lambda(x)).
\]
This shows that \eqref{eq:g_gen} holds with $\msub = \csub$.
\end{proof}
Next, we will take a look at the Clarke subdifferential of spectral functions without assuming local Lipschitzness, in order to extend 
Baes' results.
First, we will briefly explain some technical issues related to 
this task.
In Theorem~8.9 of \cite{RW}, we see that each of the generalized subdifferentials $\hat \partial, \partial, \hsub$ is associated to a corresponding notion  of normal cone. 
In this context, the Clarke subdifferential is defined using 
the convexified version of the normal cone associated to $\partial$, see Section J in chapter 8 of \cite{RW}. 
The problem is that, by doing so, the Clarke subdifferential can be larger 
than the convex hull of the approximate subdifferential. 
Therefore, in general, we have $\csub F(x) \neq \conv \partial F(x)$.

Nevertheless, under local lower semicontinuity, we have the following, see 
Lemma 4.1 in \cite{LS05}.
We recall that $f:\Re^r\to \ReIm$ is said to be \emph{locally lower semicontinuous at 
$u$}, if $f(u)$ is finite  and there exists $\epsilon > 0$ such that
$\{v \in \Re^r \mid \norm{u-v} \leq \epsilon, f(v) \leq \alpha \} $ is closed for every $\alpha$ satisfying $\alpha \leq f(u)+\epsilon$, see 
Definition~1.33 in \cite{RW}. 

\begin{lemma}[Lemma 4.1 in \cite{LS05}]\label{lem:clarke}
Suppose $f:\Re^r \to \ReIm$ is locally lower semicontinuous at $u$. Then,
\[
\csub f(u) = \closure (\conv \partial f(u) + \conv \hsub f(u)).
\]
\end{lemma}

With the aid of Lemma~\ref{lem:clarke}, we are now in position to extend Baes' results on the Clarke subdifferential.
\begin{theorem}[Clarke subgradients of  spectral functions under local lower semicontinuity]\label{theo:clarke}
	Let $F:\jAlg \to \ReIm$ be the spectral function induced by a symmetric function $f: \Re^r \to \ReIm$. The following hold:
	\begin{enumerate}[label=$(\roman*)$]
		\item 
		$F$ is locally lower semicontinuous at $x \in \jAlg$ if and only if $f$ is locally lower semicontinuous at $\lambda(x)$.
		\item If $F$ is locally lower semicontinuous at $x$, then 
		\eqref{eq:g_gen} is valid when $\msub = \csub $.
	\end{enumerate} 
\end{theorem}
\begin{proof}
Item $(i)$ follows from the continuity of the eigenvalue map $\lambda$ and 
elementary properties of the maps $\diag(\cdot,\jFr)$ and $\Diag(\cdot,\jFr)$ when $\jFr \in \jFr(x)$. We will omit its proof. 

Now, we move on to item $(ii)$.
Under Lemma~\ref{lem:clarke}, we have
\begin{align}
\csub f(\lambda(x)) & = \closure (\conv \partial f(\lambda(x)) + \conv \hsub f(\lambda(x))) \label{eq:clarke_f}\\
\csub F(x) & = \closure (\conv \partial F(x) + \conv \hsub F(x)). \label{eq:clarke_F}
\end{align}	
First, suppose that $s \in \csub F(x)$, so there is a sequence 
$\{s^k\} \subseteq \jAlg$ such that $s^k \to s$ and for each $k$ we have
\[
s^k = {\overline{s}}^k + {s^k_{\infty}},
\]
where ${\overline{s}}^k \in \conv \partial F(x) $ and 
$s^k_{\infty} \in \conv \hsub F(x)$. By Theorem~\ref{theo:hull}, there 
are $\overline \jFr^k \in \jFr(x,{\overline{s}}^k)$ and $\jFr_{\infty}^k \in \jFr(x,s^k_{\infty})$ such that
\begin{equation}
\diag({\overline{s}}^k, \overline \jFr^k ) \in \conv \partial f(\lambda(x)), \quad
\diag(s^k_{\infty},\jFr_{\infty}^k ) \in \conv \hsub f(\lambda(x)). \label{eq:clarke_aux}
\end{equation}
Because ${\overline{s}}^k$ and $s^k_{\infty}$ both operator commute 
with $x$, we conclude that $s^k$ operator commutes with $x$ as well, see \eqref{eq:simultaneous}.
Therefore, there exists a Jordan frame $\jFr^k$ such that 
$\jFr^k \in \jFr(x,s^k)$. Next, we apply Proposition~\ref{prop:diag_eig} two 
times. First with ${\overline{s}}^k$, $\jFr^k$, $\overline \jFr^k$ and then with $s^k_{\infty}$, $\jFr^k$, $\jFr^k_{\infty}$ in order to obtain that 
\begin{align}
\diag({\overline{s}}^k, \jFr^k )& \in \conv \{P\diag({\overline{s}}^k, \overline \jFr^k ) \mid  P \in \Pe{r}(\lambda(x)) \} \label{eq:clarke_aux2}\\
\diag(s^k_{\infty}, \jFr^k ) &\in \conv \{P\diag(s^k_{\infty}, \jFr^k_{\infty} ) \mid  P \in \Pe{r}(\lambda(x)) \}. \label{eq:clarke_aux3}
\end{align}
Since \eqref{eq:sym_sub} holds for the approximate and horizon subdifferentials, we have
\[
P\conv \msub f(\lambda(x)) = \conv P\msub f(\lambda(x)) = \conv \msub f(\lambda(x)),
\]
for every $P\in \Pe{r}(\lambda(x))$ when $\msub$ is $\partial$ or $\hsub$.
Therefore, \eqref{eq:clarke_aux} together with \eqref{eq:clarke_aux2} and \eqref{eq:clarke_aux3} implies that 
\[
\diag({\overline{s}}^k, \jFr^k ) \in \conv \partial f(\lambda(x)), \quad  \diag(s^k_{\infty}, \jFr^k ) \in \conv \hsub f(\lambda(x))
\]
and
\begin{equation}
\diag({\overline{s}}^k, \jFr^k ) + \diag(s^k_{\infty}, \jFr^k )  \in \conv\partial f(\lambda(x)) + \conv \hsub f(\lambda(x)). \label{eq:clarke_aux4}
\end{equation}
We now proceed as in the proof of Proposition~\ref{prop:sub_app}.
Since the idempotents in $\jFr^k$ have norm 1, passing to a converging subsequence if necessary, the Jordan frame $\jFr^k$ converges to some 
Jordan frame $\jFr \in \jFr(x,s)$ and we have
\[
\diag({\overline{s}}^k, \jFr^k ) + \diag(s^k_{\infty}, \jFr^k ) \to 
\diag(s,\jFr).
\]
Together with \eqref{eq:clarke_f} and \eqref{eq:clarke_aux4}, we conclude 
that the inclusion ``$\subseteq $'' holds in \eqref{eq:g1} when $\msub$ is $\csub$.
%
%

Now, for the ``$\supseteq$'' inclusion, suppose that $s$ is 
such that $\diag(s,\jFr) \in \csub f(\lambda(x))$ with $\jFr \in \jFr(x,s)$.
By \eqref{eq:clarke_f}, there is a sequence $\{u^k\} \subseteq \Re^r$ with 
$u^k \to \diag(s,\jFr)$ such that 
\[
u^k = \overline{u}^k + u_{\infty}^k,
\]
where $\overline{u}^k \in \conv \partial f(\lambda(x))$ and $u_{\infty}^k \in \conv \hsub f(\lambda(x))$.  Therefore, $\Diag(\overline{u}^k,\jFr) + \Diag({u}_{\infty}^k,\jFr) \to s$.
In addition, by Theorem~\ref{theo:hull}, we have
\[
\Diag(\overline{u}^k,\jFr) \in \conv \partial F(x),\quad 
\Diag({u}_{\infty}^k,\jFr) \in \conv \hsub  F(x).
\]
%
%
Using \eqref{eq:clarke_F}, we conclude that $s \in \csub F(x)$.
\end{proof}

%

\subsection{Subdifferentials of the $k$-th largest eigenvalue function}\label{sec:eig}
In this subsection, as an application of Theorems~{\ref{theo:main}}, {\ref{theo:hull}} and {\ref{theo:clarke}}, we will compute the 
generalized subdifferentials of the function $\lambda_k(\cdot):\jAlg \to \Re$ that maps an element $x \in \jAlg$ to its $k$-th largest eigenvalue, for $k \in \{1,\ldots, r\}$.

Let $f_k: \Re^r \to \Re$ be the function that maps $u \in \Re^r$ to 
its $k$-th largest component. Then, $f_k$ is a symmetric function and $\lambda _k$ is the spectral function generated by $f_k$. 
We note that, since the eigenvalue map  is Lipschitz continuous, 
each $\lambda_k$ must be Lipschitz continuous as well.
In what follows, $a^i \in \Re^r$ denotes the $i$-th unit vector and we recall that $u_i$ denotes the $i$-th component of $u\in \Re^r$. We also define
\[
\supp u \coloneqq \{ i \mid u_i \neq 0 \}.
\]
For a finite set $C$, we denote its cardinality 
by $\abs{C}$.
The generalized subdifferentials of 
$f_k$ are described by the following proposition, see Proposition 6 and Theorem 9 in \cite{Le99}.

\begin{proposition}\label{prop:max}
The following hold.
\begin{align*}
\csub f_k(u) & =  \conv \{a^i \mid f_k(u) = u_i\}, \\
\hat \partial f_k(u) & = \begin{cases}
\conv \{a^i \mid f_k(u) = u_i\}, & \text{if $k = 1$ or $f_{k-1}(u) > f_k(u)$}\\ 
\emptyset,& \text{otherwise}
\end{cases}\\
\hsub f_k(u) & = \{0 \},\\
\partial f_k(u) & = \{u \in \csub f_k(u) \mid \abs{\supp u} \leq \alpha  \},
\end{align*}
where $\alpha = 1-k + \abs{\{i \mid u_i \geq f_k(u) \} }$.
\end{proposition}
Let $\idem$ denote the set of primitive idempotents of $\jAlg$. 
We recall that $c \in \idem$ if and only if $c$ is nonzero, $\jProd{c}{c} = c$ and $c$ cannot be written 
as the sum of two nonzero orthogonal idempotents.
\begin{lemma}[Frame extension lemma]\label{lem:ext}
Let $x\in \jAlg$ and $c \in \idem$. If $\jProd{x}{c} = \sigma c$ for 
some $\sigma \in \Re$, then $\sigma$ is an eigenvalue of $x$ and 
there is a Jordan frame $\jFr \in \jFr(x)$ such that $c \in \jFr$. 
In particular, $\jFr(x,c) \neq \emptyset$.
\end{lemma}
\begin{proof}
By the Peirce decomposition (Theorem~\ref{theo:peirce1}), we have
\[
\jAlg = V(c,1) \bigoplus V\left(c,\frac{1}{2}\right) \bigoplus V(c,0).
\]
Then, since $\jProd{x}{c} = \sigma c$, we have $\jProd{c}{(x-\sigma c)} = 0$.
Therefore, $x-\sigma c \in V(c,0)$

$V(c,0)$ is a Euclidean Jordan algebra (see Theorem~\ref{theo:peirce1}).
Furthermore, since $c$ has rank $1$, the algebra $V(c,0)$ has rank $r-1$.
Therefore, we can find a Jordan frame $\hat \jFr =[c_1,\ldots, c_{r-1}]$ that diagonalizes $x-\sigma c$ in $V(c,0)$. It follows that
\begin{equation}
x = \sigma c + \sum _{i=1}^{r-1} \sigma _i c_i, \label{eq:frame_ext}
\end{equation}
where $\sigma _i \in \Re$ for every $i$.
We now need to check that $\jFr = [c,c_1,\ldots, c_{r-1}]$ is a Jordan frame.
All elements of $\jFr$ are primitive idempotents. Furthermore, $\jProd{c_i}{c_j} = 0$ if 
$i \neq j$. Since $\hat \jFr \subseteq V(c,0)$, we also have $\jProd{c}{c_i} = 0$ for every $i$.
Since the identity element of $V(c,0)$ is $\stdInt - c$ and $\hat \jFr$ is a Jordan frame in $V(c,0)$, we have 
\[
c_1 + \cdots + c_{r-1} = \stdInt - c.
\]
This shows that $c+ c_1 + \cdots + c_{r-1} = \stdInt$. Therefore, $\jFr$ is indeed a Jordan frame of the algebra $\jAlg$ and \eqref{eq:frame_ext} shows that $\jFr$ diagonalizes $x$. Since eigenvalues are unique, $\sigma$ must be one of the eigenvalues of $x$. Reordering $\jFr$ if necessary, we obtain  $\jFr \in \jFr(x,c)$.
\end{proof}

\begin{lemma}[Convex hull of primitive idempotents]\label{lem:conv}
Let $x \in \jAlg$ and $\sigma \in \Re$ be an eigenvalue of $x$. 
Let
\[
\idem(x,\sigma) \coloneqq \{c \in \idem \mid \jProd{x}{c} = \sigma c \}.
\]
Let $s \in  \conv \idem(x,\sigma)$.
\begin{enumerate}[label=$(\roman*)$]
	\item The eigenvalues of $s$ are nonnegative and sum to $1$.
	\item There is $\jFr \in \jFr (x,s)$
	such that $\inProd{s}{c} = 0$ for every $c \in \jFr$ \textbf{not} belonging to 
	$\idem(x,\sigma)$.
\end{enumerate}
\end{lemma}
\begin{proof}

Primitive idempotents have trace equal to $1$ and the trace function is linear, so elements in $\conv \idem(x,\sigma) $ must have trace $1$ too.
Then, we recall that any idempotent $c$ must be belong to $\stdCone \coloneqq \{\jProd{x}{x}\mid x \in \jAlg \}$, which is a symmetric cone (see Theorem~III.2.1 in \cite{FK94}). In particular,  $\stdCone$ is a convex cone and, since $s$ is a convex combination of elements of $\stdCone$,  $s$ belongs to $\stdCone$ which implies that its eigenvalues are nonnegative.

Next, we move on to item $(ii)$.
Pick any Jordan frame for $x$ and let $\hat c$ denote the sum of the primitive idempotents  associated to the 
eigenvalue $\sigma$. By Proposition~\ref{prop:unique_sum}, $\hat c$ does not depend on the choice 
of Jordan frame. Since $s \in \conv \idem(x,\sigma)$, we have 
\[
s = \sum _{i=1}^\ell \alpha _i c_i,
\]
where $c_i \in \idem(x,\sigma)$ for every $i$ and the $\alpha _i$ are nonnegative and sum to $1$. First, we will show that 
$s \in V(\hat c,1)$. 

By Lemma~\ref{lem:ext}, each $c_i$  can be extended to a Jordan frame $\jFr_i \in \jFr(x,c_i)$ with $c_i \in \jFr_i$. Then, 
the idempotents in $\jFr_i$ associated to the eigenvalue $\sigma$ must 
sum to $\hat c$ by Proposition~\ref{prop:unique_sum} and, at the same time, $\jProd{ c'}{c_i} = 0$ holds whenever $c' \in \jFr_i$ and $c' \neq c_i$.
We conclude that 
\[
c_i = \jProd{c_i}{c_i} = \jProd{c_i}{\sum _{c' \in \jFr_i\cap \idem(x,\sigma)  } c' } = \jProd{c_i}{\hat c}.
\]
Therefore, each $c_i$ belongs to $V(\hat c, 1)$, which 
shows that $s \in  V(\hat c, 1)$.
Since $V(\hat c,1)$ and $V(\hat c,0)$ are Euclidean Jordan algebras, 
there is a Jordan Frame $\widehat \jFr \subseteq V(\hat c, 1)$ that diagonalizes 
$s$. 
Next, since $x - \sigma\hat c \in V(\hat c, 0)$, there is a Jordan frame $\widetilde \jFr \subseteq V(\hat c, 0)$ that diagonalizes $x - \sigma\hat c$.

Let $\jFr \coloneqq \widehat \jFr \cup \widetilde \jFr$. First, 
because $\widehat \jFr \subseteq V(\hat c, 1)$ and $\widetilde \jFr \subseteq V(\hat c, 0)$ are Jordan frames, we have (the well-known fact) that $\jFr$ is a Jordan frame in the algebra $\jAlg$.

Then, since 
$\widetilde \jFr$ diagonalizes $x - \sigma \hat c$, $\widehat \jFr$ diagonalizes $s$ and the sum of the elements of 
$\widehat \jFr$ is $\hat c$ (the unit element of $V(\hat c, 1)$), we conclude that $\jFr$ diagonalizes $x$ and $s$. 
We also observe that $\widehat \jFr \subseteq \idem(x,\sigma)$, which can be seen by expressing $x$ as a linear combination of the elements in $\jFr$ and recalling that the idempotents of $\widehat \jFr$ sum to $\hat c$. 

Finally, if $c \in \jFr$ but $c \not \in \idem(x,\sigma)$ , then $c \in \widetilde \jFr$ and $\inProd{s}{c} = 0$, because 
$V(\hat c, 1)$ and $V(\hat c,0)$ are orthogonal spaces. Reordering $\jFr$ if necessary, we obtain $\jFr \in \jFr(x,s)$ with the required properties.
\end{proof}

We are now equipped to prove the following result.
\begin{theorem}[Generalized subdifferentials of $\lambda_k$]\label{theo:eig}
Let $\jAlg$ be a Euclidean Jordan algebra of rank $r$ and let $\lambda _k(\cdot)$ denote the function that maps an element to its $k$-largest eigenvalue. The following hold.
\begin{align}
\csub \lambda_k(x)& = \conv \idem(x,\lambda_k(x)) = \conv \{c \in \idem \mid \jProd{x}{c} = \lambda_k(x)c\}, \label{eq:eig_clarke}
\\
\hat \partial \lambda_k(x) & = \begin{cases}
\csub \lambda_k(x) & \text{if $k = 1$ or $\lambda_{k-1}(x) > \lambda_k(x)$}\\ 
\emptyset,& \text{otherwise}
\end{cases} \label{eq:r_gen} \\
\hsub \lambda_k(x) & = \{0 \},\\
\partial \lambda_k(x) & = \{s \in \csub \lambda_k(x) \mid {\matRank x} \leq \alpha  \},
\end{align}	
where $\alpha = 1-k + \abs{\{i \mid \lambda_i(x) \geq \lambda_k(x) \} }$.
\end{theorem}
\begin{proof}
The equality $\hsub \lambda_k(x) = \{0 \}$ follows from  Theorem~\ref{theo:main} and Proposition~\ref{prop:max}.

We will now prove the formula for $\csub \lambda _k$. 
Let $s\in \csub \lambda_k(x)$.
By Theorem~\ref{theo:hull} and Proposition~\ref{prop:max}, there exists $\jFr \in \jFr(x,s)$ such that 
\begin{equation}
\diag(s,\jFr) \in \conv \{a^i \mid \lambda _{k}(x) = \lambda _{i}(x)\}. \label{eq:s_max}
\end{equation}
Because $s$ is written as a linear combination of elements of $\jFr$, 
\eqref{eq:s_max} implies that $s$ is a convex combination of the idempotents of $\jFr$ associated to $\lambda _k(x)$. Observing  that those idempotents satisfy 
$\jProd{x}{c} = \lambda_k(x)c$, we obtain
\[
s \in \conv \{c \in \idem \mid \jProd{x}{c} = \lambda_k(x)c\},
\]
which shows that ``$\subseteq$'' holds in \eqref{eq:eig_clarke}.

%
%
Conversely, suppose that $s \in \conv \idem(x,\lambda_k(x))$. 
By item $(i)$ of Lemma~\ref{lem:conv} applied to $x,s$ and $\lambda_k(x)$, the eigenvalues of 
$s$ are nonnegative and sum to $1$. Furthermore, by item $(ii)$ of Lemma~\ref{lem:conv}, there exists $\jFr \in \jFr(x,s)$ such that 
$\inProd{s}{c} = 0$, whenever $c \in \jFr$ and $c$ is not associated to 
$\lambda _k(x)$. This, together with Proposition~\ref{prop:max}, shows that
\[
\diag(s,\jFr) \in \csub  f_k(\lambda(x))
\]because the nonzero components of $\diag(s,\jFr)$ are nonnegative, sum to $1$ and are located only at indices associated to idempotents in $\idem(x,\lambda_k(x))$.
By Theorem~\ref{theo:hull}, we have $s \in \csub \lambda _k(x)$, which 
shows that \eqref{eq:eig_clarke} holds.

The expressions for $\hat \partial \lambda _k(x), \partial \lambda_k(x) $ are consequences of Theorem~\ref{theo:main}, Proposition~\ref{prop:max}, the formula for $\csub \lambda _k(x)$ and the fact that $\abs{\supp(\lambda(x))} = \matRank(x)$.
\end{proof}

%

\section{The KL-exponent of spectral functions}\label{sec:kl}

We recall the definitions of the KL property and KL-exponent, see Definitions~2.2 and 2.3 in \cite{LP18}.
In what follows, we define $\dom \partial f \coloneqq \{u \in \Re^r \mid 
\partial f(u) \neq \emptyset \}$.
If $C$ is a subset of $\Re^r$, we define $\dist(u, C) = \inf \{\norm{v-u} \mid v \in C \}$. If $\mathcal{C}$ is a subset of $\jAlg$, we define 
$\dist(x,\mathcal{C})$ analogously using the norm induced by 
\eqref{eq:inner}.
\begin{definition}[KL-property and KL-exponent]\label{def:kl}
A lower semicontinuous function $f$ is said to satisfy the \emph{KL property at $u \in \dom \partial f$} if
there exists a neighbourhood $U$ of $u$, $\nu \in (0,\infty]$ and a continuous concave function  $\psi:[0,\nu) \to \Re_+$ with $\psi (0) = 0$ such that
\begin{enumerate}[label=$(\roman*)$]
	\item $\psi$ is continuously differentiable on $(0,\nu)$ with (its derivative) $\psi'$ positive over $(0,\nu)$;
	\item for all $v \in U$ with $f(u) < f(v) < f(u) + \nu$, we have
\[
	\psi'(f(v)-f(u))\dist(0,\partial f(v)) \geq 1.
\]
\end{enumerate}
In particular, $f$ is said to satisfy the \emph{KL property with exponent $\alpha$ at $u \in \dom \partial f$},  if $\psi$ can be taken to be $\psi(t) = ct^{1-\alpha}$ for some positive constant $c$.
\end{definition}

First, we need the following lemma.
\begin{lemma}\label{lem:dist}
Let $f:\Re^r \to \Re$ be a symmetric function and let 
$F:\jAlg \to \Re$ be the corresponding spectral function.  Then, for every $y \in \jAlg$ and for every Jordan frame $\hat \jFr$ which diagonalizes 
$y$ (see Section~\ref{sec:diag}) we have
\[
\dist(0,\partial F(y)) = \dist(0,\partial f(\diag(y,\hat \jFr))).
\]
\end{lemma}
\begin{proof}
Let  $y \in \jAlg$ and let $\hat \jFr$ be a Jordan frame which diagonalizes 
$y$. 
%
%
From \eqref{eq:sym_sub} and since permutation matrices are orthogonal 
matrices,  we obtain
\[
\dist(0,\partial f(u)) = \dist(0,\partial f(Pu)), \quad \forall u \in \Re^r, \forall P \in \Pe{r}.
\]
In particular, 
\begin{equation}\label{eq:perm_dist}
\dist(0,\partial f(\lambda(y))) = \dist(0,\partial f(\diag(y,\hat \jFr))).
\end{equation}	
Therefore, it suffices to show that $\dist(0,\partial F(y)) = \dist(0,\partial f(\lambda(y))).$
From Theorem~\ref{theo:main}, we have
\begin{align}
\dist(0,\partial F(y)) &= \min \{\norm{s} \mid \exists \jFr \in \jFr(y,s) \text{ with } \diag(s,\jFr)\in \partial f(\lambda(y))   \} \notag\\
& = \min \{\norm{\lambda(s)} \mid \exists \jFr \in \jFr(y,s) \text{ with } \diag(s,\jFr)\in \partial f(\lambda(y))\}\notag\\
& \geq \dist(0,\partial f(\lambda(y))). \notag
\end{align}	
Therefore, $\dist(0,\partial F(y)) \geq \dist(0,\partial f(\lambda(y)))$. 
To show the opposite inequality, let $ d \in \partial f(\lambda(y)), \jFr \in \jFr(y).$
By Theorem~\ref{theo:main}, $s \coloneqq \Diag(d,\jFr)$ is such 
that $s \in \partial F(y)$. Furthermore, we have $\norm{s} = \norm{d}$.
This shows that $\dist(0,\partial F(y)) \leq \dist(0,\partial f(\lambda(y)))$.
\end{proof}

\begin{theorem}[Transfer principle for the KL property and KL exponent]\label{theo:kl}
	Let $f:\Re^r \to \Re$ be a symmetric function and let 
	$F:\jAlg \to \Re$ be the corresponding spectral function. Then, 
	\begin{enumerate}[label=$(\roman*)$]
		\item $F$ satisfies the KL property $x$ if and only if $f$ satisfies the KL property at $\lambda(x)$. In addition, the $\psi$ and $\nu$ in Definition~\ref{def:kl} can be taken to be the same for both $f$ and $F$.
		\item $F$ satisfies the KL property with exponent $\alpha$ at $x$ if and only if $f$ satisfies the KL property with exponent $\alpha$ at $\lambda(x)$.	
	\end{enumerate}

\end{theorem}
\begin{proof}
		First we prove item $(i)$.	
		By Theorem~\ref{theo:main} we have $x \in \dom \partial F$ if 
		and only if $\lambda(x) \in \dom \partial f$.
		Next, suppose that $f$ satisfies the KL property at $\lambda(x)$
		and let $U, \nu$ and $\psi$ be as in Definition~\ref{def:kl}.
		
		Since $\lambda$ is continuous, $\mathcal{U} \coloneqq \lambda^{-1}(U)$ is a neighbourhood of $x$. Therefore, if $y \in \mathcal{U}$ is 
		such that $F(x) < F(y) < F(x) + \nu$, we have 
\[
\lambda(y) \in U \text{ and } f(\lambda(x)) < f(\lambda(y)) < f(\lambda(x)) + \nu.
\]
		By Lemma~\ref{lem:dist} and item $(ii)$ of Definition~\ref{def:kl} applied to $f$ and $\psi$, we have
\[
		\psi'(F(y)-F(x))\dist(0,\partial F(y))= \psi'(F(y)-F(x))\dist(0,\partial f(\lambda(y))) \geq 1.
\]
		This shows that $F$ satisfies the KL property at $x$ with the same 
		$\psi$ and $\nu$.
		
		Now, we prove the converse. Suppose that $F$ satisfies 
		the KL property at $x$ and let $\mathcal{U}$ be a neighbourhood of $x$ together with $\psi$ and  $\nu$ such that 
			Definition~\ref{def:kl} is satisfied.
			
		Let $\jFr \in \jFr(x)$ and 
		$U \coloneqq \Diag(\cdot, \jFr)^{-1}(\mathcal{U})$.
		Then, whenever $v \in U$ is such that $f(\lambda(x)) < f(v) < f(\lambda(x)) + \nu$, we have
\[
\Diag(v,\jFr) \in \mathcal{U} \text{ and } F(x) < F(\Diag(v,\jFr)) < F(x) + \nu.
\]
		By item $(ii)$ of Definition~\ref{def:kl}, we have
\[
		\psi'(f(v)-f(\lambda(x)))\dist(0,\partial F(\Diag(v,\jFr))) \geq 1.
\]
		By Lemma~\ref{lem:dist}, we have 
\[
		\psi'(f(v)-f(\lambda(x)))\dist(0,\partial f(v)) \geq 1.
\]
		This shows that $f$ satisfies the KL property at $\lambda(x)$ with the same 
		$\psi$ and $\nu$, which concludes the proof of item $(i)$.
		
		Next, we observe that item $(ii)$ is a particular case of the previous item, when 	$\psi$ can be taken to be $\psi(t) = ct^{1-\alpha}$. 
\end{proof}

\begin{remark}\label{rem:comp}
	In Theorem~3.2 of \cite{LP18} there is a result about the KL-exponent of function compositions of the form $g_1(g_2(\cdot))$. 
	However, the result requires that $g_2$ be continuously differentiable, so it cannot be used to prove Theorem~\ref{theo:kl}. 
\end{remark}
{
	\small{
\section*{Acknowledgments}
We thank the referees for their  comments, which helped to improve the 
paper. This work was partially supported by the Grant-in-Aid for Scientific Research (B) (19H04069) and  the Grant-in-Aid for Young Scientists (19K20217) from Japan Society for the Promotion of Science.
}
}
\bibliographystyle{abbrvurl}
\bibliography{bib}
\end{document}